\documentclass[a4paper,11pt]{amsart}

\usepackage{color}
\usepackage{amsxtra}
\usepackage{amssymb}
\usepackage{mathtools}
\usepackage{enumerate} 
\mathtoolsset{showonlyrefs} %	enumera solo las ecuaciones que se citan
\usepackage{ mathrsfs }
\usepackage{tikz}

\usetikzlibrary{positioning, calc}
\usetikzlibrary{shadings}
\usepackage{caption}
\usepackage{subcaption}
\usepackage{color}
\usepackage{hyperref}
\usepackage{pgfplots}
\usetikzlibrary{fadings}
\usepackage{tikz}
\hypersetup{
  colorlinks   = true, %Colours links instead of ugly boxes
  urlcolor     = blue, %Colour for external hyperlinks
  linkcolor    = blue, %Colour of internal links
  citecolor   = red %Colour of citations
}
\newcommand{\kom}[1]{}
\renewcommand{\kom}[1]{{\bf [#1]}}

\theoremstyle{plain}
\newtheorem{theorem}{Theorem} [section]
\newtheorem{corollary}[theorem]{Corollary}
\newtheorem{lemma}[theorem]{Lemma}
\newtheorem{proposition}[theorem]{Proposition}
\newtheorem{example}[theorem]{Example}

\theoremstyle{definition}
\newtheorem{definition}[theorem]{Definition}
\newtheorem{remark}[theorem]{Remark}

\numberwithin{theorem}{section}
\numberwithin{equation}{section}
\newcommand{\R}{{\mathbb R}}

 \newcommand{\eps}{{\varepsilon}}
 \def\1{\raisebox{2pt}{\rm{$\chi$}}}

\usepackage{enumerate}
\newcommand{\Om}{\Omega}

\newcommand{\abs}[1]{\left|#1\right|}

\newcommand{\Rn}{\mathbb{R}^N}
\newcommand{\dist}{\operatorname{dist}}

\title[A bridge between convexity and quasiconvexity]{
A bridge between convexity and quasiconvexity
}
\author[Blanc]{Pablo Blanc}
\address{Departamento  de Matem{\'a}tica, FCEyN, Universidad de Buenos Aires, Pabellon I,
         Ciudad Universitaria (C1428BCW), Buenos Aires, Argentina.}
\email{pblanc@dm.uba.ar}

\author[Parviainen]{Mikko Parviainen}
\address{Department of Mathematics and Statistics, University of
Jyv\"askyl\"a, PO~Box~35, FI-40014 Jyv\"askyl\"a, Finland}
\email{mikko.j.parviainen@jyu.fi}

   \author[Rossi]{Julio Rossi}
\address{ Departamento  de Matem{\'a}tica, FCEyN, Universidad de Buenos Aires, Pabellon I,
         Ciudad Universitaria (C1428BCW), Buenos Aires, Argentina.}
\email{jrossi@dm.uba.ar}

\begin{document}


\begin{abstract} We introduce a notion of convexity with respect to a one-dimensional operator 
and with this notion find
a one-parameter family of different convexities that interpolates between
classical convexity and quasiconvexity.  
We show
that, for this interpolation family, the convex envelope of a continuous boundary datum in a strictly convex domain 
is continuous up to the boundary and is characterized as being the unique viscosity solution to the
Dirichlet problem in the domain for
a certain fully nonlinear partial differential equation that involves the associated operator. 
In addition we prove that the convex envelopes of a boundary datum constitute a one-parameter curve of functions that goes from the quasiconvex envelope to the convex envelope being continuous with respect to uniform convergence. 
Finally, we also
show some regularity results for the convex envelopes proving that there is 
an analogous to a supporting hyperplane at every point and that convex envelopes are
$C^1$ if the boundary data satisfies in particular $NV$-condition we introduce. 
\end{abstract}

\maketitle

%{\small
%\tableofcontents}
   
\section{Introduction}
	
	The main goal of this paper is to build a bridge connecting the notions of convexity
and quasiconvexity for functions defined in the Euclidean space.
We begin by recalling the basic definitions of convex and quasiconvex functions.
Then, we introduce the notion of convexity with respect to an operator.
This notion allows us to build the bridge by presenting a one-parameter family of operators so that the notion of convexity with respect to them interpolates between convexity and quasiconvexity.

	\subsection{Classical convexity}  
	First, let us recall the usual notion of convexity.  Let $\Omega \subset \Rn$ be a convex domain. 
	A function $u : \overline{\Omega} \to  {\mathbb{R}}$ is said to be convex in $\overline{\Omega}$  if
	for any two points $x,y \in \overline{\Omega}$ it holds that
	\begin{equation} \label{convexo-usual}
	u(tx+(1-t)y) \leq tu(x) + (1-t) u(y), \qquad \text{ for all } t \in (0,1).
	\end{equation}
	We refer to \cite{Vel} for a general reference on convex structures.

	Notice that $v(t) = tu(x) + (1-t) u(y)$ is just the solution to the equation $v'' (t)=0$ 
	in the interval $(0,1)$ that verifies $v(1) = u(x) $ and $v(0) = u(y)$ at the endpoints.
	Therefore, one can rewrite \eqref{convexo-usual} as 
	\begin{equation} \label{convexo-usual.33}
	u(tx+(1-t)y) \leq v(t), \qquad \text{ for all } t \in (0,1),
	\end{equation}
	with $v(t)$ the solution to 
	$$
	\left\{
	\begin{array}{l}
\displaystyle v'' (t) =0, \qquad t \in (0,1), \\
\displaystyle v(0) = u(y), \\ 
\displaystyle v(1) = u(x).
\end{array}
\right.
	$$

	Given a boundary datum $g : \partial \Omega \to  {\mathbb{R}}$ one can define the convex envelope of
	$g$ inside $\Omega$ as the largest convex function that is below $g$ on $\partial \Omega$, that is, we take
	\begin{equation} \label{convex-envelope-usual}
	u^* (z) = \sup \Big\{v(z) : v \mbox{ is convex in $\overline{\Omega}$ and verifies } v|_{\partial \Omega} \leq g \Big\},
	\qquad z \in \Omega.
	\end{equation}
	When $\Omega$ is bounded and strictly convex and the boundary datum $g$ is continuous, the convex envelope $u^*$ is continuous in $\overline{\Omega}$, see \cite{BlancRossi}.
	Recall that $\Omega$ is strictly convex if $xt+y(1-t)\in \Omega$ for every $x,y\in\overline{\Omega}$ and $t\in (0,1)$.
	The convex envelope can be characterized as the
	unique solution (the equation has to be
	interpreted in viscosity sense) to
	\begin{equation} \label{convex-envelope-usual-eq}
	\left\{\begin{array}{ll}
	\displaystyle \lambda_1 (D^2 u) (z) := \inf_{|v|=1} \langle D^2 u (z) v, v \rangle = 0, \qquad & z\in  \Omega, \\
	u(z) = g (z),  \qquad & z\in  \partial \Omega,
	\end{array} \right.
		\end{equation}
		see \cite{OS,Ober}.
	Here $\lambda_1 (D^2 u)$ is the smallest eigenvalue of
	the Hessian matrix, $D^2u$. We also refer to \cite{BlancRossi,GR,HL1,KK,OS,Ober,Ober2} for extra information
	and applications of these results. 	
	For a fractional version of these ideas, see \cite{DpQR}.

	\subsection{Quasiconvexity}  
	A notion weaker than convexity is quasiconvexity. A function 
	$u\colon \overline{\Omega}\to \mathbb{R} $ is called quasiconvex 
	if for all $x,y\in \overline{\Omega}$ and any $t \in (0,1),$ 
	we have
	\begin{equation} \label{quasiconvexo-usual}
		u(t x+(1- t )y)\leq \max \Big\{ u(x),u(y)\Big\}.
	\end{equation}
	An alternative and more geometric way of defining a quasiconvex function 
	$u$ is to require that each sublevel set $S_{\lambda}(u)=\{z\in \overline{\Omega} \, : \, u(z)\leq \lambda\}$
	is a convex set in $\mathbb{R}^N$. 
	Quasiconvex functions have applications in mathematical analysis, 
	optimization, game theory, and economics. See, for example, \cite{CoSa,DiGuglielmo,Pearce,Ko,Sion}
	and \cite{10} and references therein for an overview.

	Now, going back to \eqref{quasiconvexo-usual}, we notice that $v(t) = \max \left\{ u(x),u(y)\right\}$ is the viscosity solution to the equation $|v' (t)|=0$ 
	in the interval $(0,1)$ with boundary conditions $v(1) = u(x) $ and $v(0) = u(y)$, see Remark~\ref{rem:max} below.
	Therefore, one can rewrite \eqref{quasiconvexo-usual} as 
	\begin{equation} \label{quasiconvexo-usual.33}
	u(tx+(1-t)y) \leq v(t), \qquad \text{ for all } t \in (0,1),
	\end{equation}
	with $v(t)$ the solution to 
	$$
	\left\{
	\begin{array}{l}
\displaystyle |v'(t)| =0, \qquad t \in (0,1), \\
\displaystyle v(0) = u(y), \\ 
\displaystyle v(1) = u(x).
\end{array}
\right.
	$$
	
	Associated with
	the quasiconvex envelope of a boundary datum $g:\partial \Omega \to \mathbb{R}$,
	that is defined as
		\begin{equation} \label{quasiconvex-envelope-usual}
	u^* (z) = \sup \Big\{v(z) : v \mbox{ is quasiconvex in $\overline{\Omega}$ and verifies } v|_{\partial \Omega} \leq g \Big\}, \quad  z \in \Omega,
	\end{equation}
	there is a partial differential equation. 
	 In fact, the quasiconvex envelope is characterized as the unique
	 quasiconvex viscosity solution to  
	\begin{equation}\label{ec-RN}
		 \min_{ \substack{|v|=1, \\
		\langle v,  D  u(x) \rangle =0}} 
		\langle D^2 u(x) v, v \rangle =0.
	\end{equation}
 This equation is studied in \cite{BGJ12a,BGJ12b,BGJ13}. 
	In particular, in
	\cite{BGJ13} it is proved that there is no uniqueness for general viscosity solutions, but nonetheless, there is 
	uniqueness among quasiconvex solutions.

\subsection{Convexity with respect to an operator}	
	Now, let us introduce a general definition of convexity that will be the key to build 
	a family of notions interpolating between convexity and quasiconvexity. 
	
		We fix a 1-dimensional operator $L=L(v'',v')$ and
	we propose the following definition for convexity with respect to the operator 
	$L$.

\begin{definition} \label{L-convex-defi}	
	A function $u :\overline{\Omega} \to  {\mathbb{R}}$ is said to be $L$-convex if
	for any two points $x,y \in \overline{\Omega}$ it holds that
	\begin{equation} \label{convexo-X}
	u(tx+(1-t)y) \leq v(t), \qquad \text{ for all } t \in (0,1),
	\end{equation}
	where $v$ is just the solution to 
	$$
	\left\{
	\begin{array}{l}
\displaystyle L (v) (t) =0, \qquad t \in (0,1), \\
\displaystyle v(0) = u(y), \\ 
v(1) = u(x).
\end{array}
\right.
	$$
	\end{definition}
	
Notice that this definition works well when we have solvability of the one-dimensional 
Dirichlet problem for the operator $L$ in the interval $(0,1)$.
Remark that  the usual notions of convexity and quasiconvexity 
fit into this definition. In fact, as we have mentioned, 
with the choice $L(v) = v''$ we recover the usual notion of convexity and with $L(v) = |v'|$ we obtain
quasiconvexity.

	\subsection{A bridge between convexity and quasiconvexity.} 
		We consider the one-parameter family of operators
	\begin{equation} \label{familia}
	L_\alpha (v) (t) = \alpha v'' (t) + (1-\alpha) |v' (t)|^2,
	\end{equation}
	with $0\leq \alpha \leq 1$ in Definition \ref{L-convex-defi}.
	Notice that for $\alpha=0$ we have quasiconvexity and for $\alpha=1$ we obtain convexity. 
	We call a function $u : \overline{\Omega} \to \mathbb{R}$ {\emph{$\alpha$-convex}} if it is $L_\alpha$-convex in the sense of 
	Definition \ref{L-convex-defi}.

	Notice that this operator $L_\alpha$ behaves well
	under the scaling $w(t) = v(Rt)$, in fact we have
	$$
	\begin{array}{l}
	\displaystyle 
	L_\alpha (w) (t) = \alpha w'' (t)+ (1-\alpha) |w' (t)|^2 \\[6pt]
	\qquad \qquad  = R^2  \alpha v'' (Rt) + (1-\alpha) |v' (Rt)|^2R^2 = R^2 
	L_\alpha (v) (R t).
	\end{array}
	$$
	This fact is the key to obtain the second part of
	our first theorem that gives a characterization of being $\alpha$-convex in $\overline{\Omega}$
	in terms of a PDE.
	
	\begin{theorem} \label{teo.1.intro} Let $\alpha \neq 0$. 
	Then, an $\alpha$-convex function $u:\overline{\Omega} \to \mathbb{R}$  is
	Lipschitz continuous in $\Omega$.
	Moreover, a function $u:\overline{\Omega} \to \mathbb{R}$ 
	is $\alpha$-convex if and only if it satisfies
	\begin{equation} \label{convex-envelope-X}
	\mathcal{L}_\alpha  u (z) := \inf_{
	\substack{x,y \in \overline{\Omega}  \\
		z=t_0 x+ (1-t_0)y} } L_\alpha (u) (t_0) \geq 0, \qquad z \in \Omega,
		\end{equation}
		in the viscosity sense.
		Here $L_\alpha (u) (t_0)$ stands for $L_\alpha$ applied to $u(tx+(1-t)y)$ as a function of 
		$t\in (0,1)$ at the point $t_0$.
		\end{theorem}
		
		In the previous theorem the inequality has to be
	interpreted in the viscosity sense  touching with $1-$dimensional test functions as explained in Definition \ref{defi-visc}.

	Observe that the result does not hold for $\alpha =0$ (the quasiconvex case) since
	there are quasiconvex functions that are discontinuous. Moreover,
	as was pointed out in \cite{BJ}, it holds that
		$$
		u(z_1,z_2) = -(z_1)^4
		$$ 
		is a solution to the associated PDE, \eqref{ec-RN}, that is not quasiconvex.

In Section~\ref{sect-L} we prove Theorem~\ref{teo.1.intro} and obtain other results concerning $\alpha$-convex functions.
Among them we prove an analogous to the well known fact that the graph of a convex function can be touched from below with a supporting hyperplane in the context of 
$\alpha$-convex functions, see Theorem~\ref{chess}.  

We also establish in Proposition \ref{prop.compos} that $\alpha$-convex functions can be decomposed by suitable monotone functions. This is related to the well known  fact that convexity and quasiconvexity differ in how they behave under monotone transformations.
 
In Section \ref{sec:basic}, we derive some basic properties of $\alpha$-convex functions. We observe that notions of convexity associated with 
	$L_\alpha$ are stronger as $\alpha$ increases. This is naturally consistent with the fact that convexity (the concept for $\alpha=1$) implies quasiconvexity
	($\alpha=0$).
		We also observe that the supremum of $\alpha$-convex functions is $\alpha$-convex, and that $\alpha$-convex function attains its infimum under suitable conditions.

	\subsection{The $\alpha$-convex envelope} 
	Using the definition of $\alpha$-convexity one can define the $\alpha$-convex envelope of a boundary datum $g$.

\begin{definition}
\label{def:convex-env}
Let $g : \partial \Omega \to  {\mathbb{R}}$ be a continuous function.
Then, the $\alpha$-convex envelope of $g$ in $\Om$ is given by
	\begin{equation} \label{convex-envelope-X.44}
	u_\alpha^* (z) = \sup \Big\{v(z) : v \mbox{ is $\alpha$-convex and verifies } v|_{\partial \Omega} \leq g 			\Big\}, \quad z\in \Omega.
		\end{equation} 
\end{definition}
%%%%%%%%%%
	Notice that the definition of the convex envelope 
	makes sense even when $g$ is not continuous, for example
	we can ask for $g$ bounded below, or when the domain is not strictly convex.  However, we aim at proving regularity results that require a continuity assumption and the strict convexity of the domain in order to have a convex envelope that attains the boundary condition with continuity. Notice that for a non-strictly convex domain, the convex envelope may be discontinuous even when $g$ is continuous (see Example 12 in \cite{BlancRossi}).

	In an interval $(a,b)$, the $\alpha$-convex envelope is explicit, it
is just the solution to the equation $L_\alpha (u)=0$ with boundary data $g(a)$, $g(b)$. Thus we assume $N\geq 2$ in what follows.
	
	\begin{theorem} \label{thm:cont} 
	Assume that $\Omega\subset\R^N$ is a strictly convex bounded domain, $N\geq 2$, and that $g:\partial \Omega \to \mathbb{R}$ is continuous. Then, the family of $\alpha$-convex envelopes $\{u^*_\alpha\}_{\alpha\in [0,1]}$ 
	is equicontinuous in $\overline{\Omega}$.
		\end{theorem}

	  	Observe that the operator $\mathcal{L}_\alpha$ that appears in Theorem \ref{teo.1.intro} can be written as 
	 \begin{equation} \label{form-alter}
	 \begin{array}{l}
\displaystyle \mathcal{L}_\alpha  u (z) := \inf_{
	\substack{x,y \in \overline{\Omega}  \\
		z=t_0 x+ (1-t_0)y} } L_\alpha (u) (t_0)  \\[6pt]
\qquad \qquad \displaystyle =  \inf_{|v|=1}\Big\{ 
		\alpha \langle D^2 u(z) v, v \rangle + (1-\alpha)  |\langle D u(z), v \rangle |^2
		\Big\} , \qquad z \in \Omega.
		\end{array}
		\end{equation}
	We use this formulation here since it streamlines the exposition.	
	\begin{theorem} \label{teo.2.intro} 
	Let $\alpha\neq 0$. Assume that $\Omega$ is a strictly convex bounded domain and
	that $g:\partial \Omega \to \mathbb{R}$ is continuous. Then, the $\alpha$-convex envelope
	is continuous in $\overline{\Omega}$ and 
	is characterized as the unique viscosity solution to 
	\begin{equation} \label{eq.main.intro}
	\left\{
	\begin{array}{ll} \! \! 
	\displaystyle \mathcal{L}_\alpha u (z) \! = \!\!
	 \inf_{|v|=1} \! \! 
		\Big\{\alpha \langle D^2 u(z) v, v \rangle \! + \! (1-\alpha)  |\langle D u(z) ,v \rangle |^2 \Big\} \! = \!  0,  & z \!  \in \!  \Omega, \\[8pt]
	\! \!  u (z) = g (z),  & z \!  \in \!  \partial \Omega.
	\end{array}
	\right.
		\end{equation}
		\end{theorem}
	It is also worth noting that the definition of viscosity solutions in the theorem above is nonstandard, and uses one dimensional test functions for subsolutions. However, we later observe in Remark~\ref{rem:conincide} that such a definition coincides with the usual definition with $N$-dimensional test functions.
				
Recall that, as we have already mentioned, for the quasiconvex envelope
(that corresponds to $\alpha=0$), the PDE does not have a unique solution. 
In this case it is known that, for
	$\Omega$ a strictly convex bounded domain and
	$g:\partial \Omega \to \mathbb{R}$ a continuous function,
	then, the quasiconvex envelope
	is characterized as the unique quasiconvex viscosity solution to 
		\begin{align*}
\begin{cases}
	\displaystyle \min_{ \substack{|v|=1, \\
		\langle v,  D  u(x) \rangle =0}} 
		\langle D^2 u(x) v, v \rangle =0,  & z \in \Omega, \\
		 u (z) = g (z),  & z \in \partial \Omega.
\end{cases}
\end{align*}
See Theorem~5.5 in \cite{BGJ13}.
Moreover, a comparison principle holds (Corollary 5.6 in \cite{BGJ13}) if subsolutions are required to be quasiconvex.

	Our next result says that the $\alpha$-convex envelopes of a fixed boundary continuous datum $g:\partial \Omega \to \mathbb{R}$
	inside $\Omega$ is a continuous one-parameter curve that goes from the quasiconvex envelope 
	at $\alpha=0$ to the usual convex envelope at $\alpha =1$. 
	Therefore, this notion of $\alpha$-convexity allows us to construct a 
	continuous bridge that has quasiconvexity and convexity at its endpoints.

\begin{theorem} \label{th.continu} 
	Assume that $\Omega\subset\R^N$ is a strictly convex bounded domain, $N\geq 2$ and that $g:\partial \Omega \to \mathbb{R}$ is continuous.
Then, the map $$\alpha \mapsto u^*_\alpha$$ is non-increasing, and
continuous with respect to the $\sup$-norm, that is, as $\alpha \to \hat \alpha\in [0,1]$, it holds that 
$$
u^*_\alpha\to u^*_{\hat \alpha}
$$
 uniformly in $\overline{\Omega}$. 
\end{theorem}

 Observe that, in particular, we have the convex envelope
at one end
$$
u^*_\alpha \to u_1^*, \qquad \mbox{ as } \alpha \to 1,
$$
and the quasiconvex envelope at other end of the curve,
$$
u^*_\alpha \to u_0^*, \qquad \mbox{ as }\alpha \to 0.
$$

	We study $\alpha$-convex envelopes in Section \ref{sect-envelope}. We prove that the envelope is a solution
to the PDE; in Section \ref{sect-bridge} we prove Theorem \ref{th.continu}.

Finally, in the last section, Section \ref{sect-regul},
we prove a $C^1$ regularity result for the 
$\alpha-$convex envelope for $\alpha\neq 0$.  
This result needs a technical assumption on the boundary datum that
we call NV (see Definition \ref{NV-property} in Section \ref{sect-regul}). 
The idea behind this condition is to prevent a wedge-like behavior propagating from the boundary inside the domain. We observe that when $\partial \Omega$
is $C^1$ and $g$ is $C^1$ condition NV holds, but more general situations are possible.

Our regularity result reads as follows.

\begin{theorem} \label{teo.C1.intro} Let $\alpha\neq 0$.
If $\Omega$ is strictly convex and the boundary data $g$ is continuous and 
satisfies NV, then the $\alpha$-convex envelope $u^*_\alpha$ is $C^1(\Omega)$.
\end{theorem}

Concerning $C^1$ regularity for the usual convex envelope, condition NV is more
general than the smoothness assumptions in  \cite{OS} by Oberman and Silvestre but in that reference the authors prove 
$C^{1,\alpha}$-regularity.

\section{$\alpha$-convex functions} \label{sect-L}

In this section we begin by studying the $1-$dimensional problem associated with 
the operator $L_\alpha (v) (t) = \alpha v'' (t) + (1-\alpha) |v' (t)|^2$ for $\alpha\in (0,1)$.
Note that for $\alpha=0$ and for $\alpha=1$ we already know the explicit solution.

\subsection{1-dimensional equation}

We begin by computing the explicit solution of the 1-dimensional equation.
Below we assume $b>a$, since $a>b$ is analogous and $a=b$ gives a constant solution.

\begin{lemma} \label{lemma-L-1d}
Let $\alpha\in (0,1)$ and $b>a$. The solution to 
\begin{equation} \label{Mikko}
\left\{
\begin{array}{l}
\displaystyle L_\alpha (v) (t) = \alpha v'' (t) + (1-\alpha) |v' (t)|^2 =0, \qquad t \in (0,1), \\
\displaystyle v(0) = a, \\ 
v(1) = b,
\end{array}
\right.
\end{equation}
is given by
\begin{equation}
\label{eq:v}
	v(t)=\frac{1}{K_\alpha}\ln(1+(e^{(b-a)K_\alpha}-1)t)+a,
\end{equation}
	with $K_\alpha = \frac{(1-\alpha)}{\alpha}$.
	\end{lemma} 
\begin{proof}
	Call $\gamma (t) = v'(t)$, then $\gamma$ is a solution to 
	$$
	\gamma' (t) + \frac{(1-\alpha)}{\alpha} |\gamma (t)|^2 =0, \qquad t \in (0,1).
	$$
	A simple integration of this equation gives that
	$$
	\gamma (t) = \frac{1}{ C_1 + K_\alpha t} 
	$$
	with $K_\alpha = \frac{(1-\alpha)}{\alpha}$ and $C_1$ a constant. 
	
	Then, the solutions to $\alpha v'' (t) + (1-\alpha) |v' (t)|^2 =0$ are given by
	$$
	v(t) = \frac{1}{K_\alpha } \ln ( C_1 + K_\alpha t ) + C_2.
	$$
	Now, we only have to choose $C_1$ and $C_2$ in order to fulfil $v(0) = a$ and $v(1) = b$. We obtain, after some computations,
$$
	C_1=\frac{K_\alpha}{e^{(b-a)K_\alpha}-1} \quad\text{and}\quad C_2=a-\frac{\ln(C_1)}{K_\alpha}.
	$$
Finally, we can use these constants to rearrange the expression for $v$ to get \eqref{eq:v}.
\end{proof}

Now, we state some consequences that follow from the previous explicit formula.

\begin{remark}
\label{rem:scaling}
	The problem behaves well under scaling, that is,
	$w(t) = v(Rt)$ is the solution to 
	$$
	\left\{
	\begin{array}{l}
\displaystyle L_\alpha (w) (t) = \alpha w'' (t) + (1-\alpha) |w' (t)|^2 =0, \qquad t \in (0,1/R), \\
\displaystyle w(0) = a, \\ 
w(1/R) = b.
\end{array}
\right.
	$$
	Therefore, given $x,y\in\Omega$ and $t_0<t_1$ we have  
		\begin{equation}
	u\left(\frac{t-t_0}{t_1-t_0}x+\frac{t_1-t}{t_1-t_0}y\right) \leq v(t), \qquad \text{ for all } t \in (t_0,t_1),
	\end{equation}
	where $v$ is the solution to 
	$$
	\left\{
	\begin{array}{l}
\displaystyle L v (t) =0, \qquad t \in (t_0,t_1), \\
\displaystyle v(t_0) = u(y), \\ 
v(t_1) = u(x).
\end{array}
\right.
	$$
\end{remark}

Concerning the comparison principle for $L_{\alpha}$ we have the following lemma.

\begin{lemma}
\label{lem:1dcomp}
For $\alpha\neq 0$, the operator $L_{\alpha}$ has a comparison principle: if the boundary values  are ordered, so are the solutions.	Even more, $L_\alpha$ has a strong comparison principle, if $v_1$ and $v_2$
	are two solutions to \eqref{Mikko} such that $v_1\geq v_2$ in $(0,1)$ and 
	they touch at one point, $v_1 (t_0) = v_2(t_0)$ then $v_1$ and $v_2$ coincide
	in the whole $(0,1)$, 
	$$
	v_1(t) = v_2(t), \qquad t \in (0,1).
	$$
	\end{lemma}
	\begin{proof}
	The comparison principle follows from the explicit formula for the solution to the equation. We have to prove that we have a strong comparison principle.
	Observe that solutions 
	are smooth, therefore, if we have two solutions
	$v_1$ and $v_2$
	to \eqref{Mikko} such that $v_1\geq v_2$ in $(0,1)$ and 
	they touch at one point, $v_1 (t_0) = v_2(t_0)$ 
	then we also have  $v_1'(t_0) = v_2'(t_0)$. Then, we conclude that
	$$
	v_1(t) = v_2(t), \qquad t \in (0,1),
	$$
	from the uniqueness of solutions to the second order ODE
	$$
	\alpha v'' (t) + (1-\alpha) |v' (t)|^2 =0,
	$$
	with initial conditions at $t_0$, $v(t_0)=c_1$, $v'(t_0)=c_2$, 
	(see Theorem 7 in Section 8 of \cite{Codin}). 
	\end{proof}
	
	\begin{remark} \label{rem:continuity1-d} 
	From the explicit expression we obtain that $v_\alpha (t)$ is continuous with respect to 
	the boundary values $a$, $b$ and also with respect to $\alpha \in (0,1]$. Therefore, if we have sequences 
	$a_n\to a$, $b_n \to b$ and $\alpha_n \to \alpha \neq 0$ the corresponding solutions
	verify
	$$
	v_n (t) \to v(t)
	$$
	uniformly for $t \in [0,1]$.
	For $\alpha =0$ we have the pointwise convergence but it is not necessarily uniform. 
	 \end{remark}

	\begin{remark}
	\label{rem:delta}
	Also notice that since the explicit expression involves a 
	logarithm,  
	$$
	v(x)=\frac{1}{K_\alpha}\ln(1+(e^{(b-a)K_\alpha}-1)t)+a,
	$$
	we can extend $v$ to a maximal interval of the form $(-\delta , +\infty) $ when $b>a$, or of the form $(-\infty, \delta)$ when $a>b$ and when $a=b$ the solution is constant and the maximal interval where it is defined is $(-\infty,+\infty)$.
	Here
	$\delta$ depends on $|b-a|$ and $\alpha$, and it is given by 
	$$
	\delta (|b-a|, \alpha) = \frac{1}{e^{|b-a|K_\alpha}-1}.
	$$
	Notice that $\delta \to 0$ when $\alpha \to 0$ and $\delta \to \infty$ when $\alpha \to 1$ or when
	$|b-a|\to 0$.  
	\end{remark}

\begin{remark}
\label{rem:outside}
	We have stated as part of our definition of being $\alpha$-convex that if $u$ is $\alpha$-convex, then 
		\begin{equation}
	u\left(\frac{t-t_0}{t_1-t_0}x+\frac{t_1-t}{t_1-t_0}y\right) \leq v(t), \qquad \text{ for all } t \in (t_0,t_1),
	\end{equation}
	where $v$ is the solution to 
	$$
	\left\{
	\begin{array}{l}
\displaystyle L v (t) =0, \qquad t \in (t_0,t_1), \\
\displaystyle v(t_0) = u(y), \\ 
v(t_1) = u(x).
\end{array}
\right.
	$$
	
We also have that $$u\left(\frac{t-t_0}{t_1-t_0}x+\frac{t_1-t}{t_1-t_0}y\right)  \geq v(t)$$ for every  $t \in I\setminus(t_0,t_1)$ where $I$ is the maximal interval where both functions are defined. 
In fact,
suppose, for the sake of contradiction, that it does not hold.
Then, there exist $t_2$ such that $$u\left(\frac{t_2-t_0}{t_1-t_0}x+\frac{t_1-t_2}{t_1-t_0}y\right)  < v(t_2).$$
Assume without loss of generality that $t_1<t_2$ and consider $w$ such that
	$$
	\left\{
	\begin{array}{l}
\displaystyle L w (t) =0, \qquad t \in (t_0,t_2), \\
\displaystyle w(t_0) = u(y), \\ 
w(t_2) = u\left(\frac{t_2-t_0}{t_1-t_0}x+\frac{t_1-t_2}{t_1-t_0}y\right).
\end{array}
\right.
	$$
	Then, by the comparison principle, Lemma~\ref{lem:1dcomp}, we have $v>w$ in $(t_0,t_2)$.
	And we get $u(x)=v(t_1)>w(t_1)$ which is a contradiction since $u$ is $\alpha$-convex.
\end{remark}

Finally we need to consider an approximate solution to the problem.
We compute the explicit solution to the equation $\alpha w'' (t) + (1-\alpha) |w' (t)|^2 =-\eta^2$
with $\eta$ small.
We include the details in the case $\alpha \in (0,1)$ and we leave the cases $\alpha=0$ 
and $\alpha =1$ to the reader.

\begin{lemma} \label{lem:1d-eta}
Let $\alpha\in (0,1)$, $b>a$ and $\eta>0$ small enough.
The solution to 
\begin{equation} 
\label{eqwitheta}
\left\{
\begin{array}{l}
\displaystyle L_\alpha w (t) = \alpha w'' (t) + (1-\alpha) |w' (t)|^2 =-\eta^2, \qquad t \in (0,1), \\
\displaystyle w(0) = a, \\ 
w(1) = b,
\end{array}
\right.
\end{equation}
is given by
	$$
	w(t) = \frac{1}{K_\alpha } \ln \left( \cos\left(C_1 + \frac{\eta\sqrt{(1-\alpha)}}{\alpha} t\right) \right) + C_2.
	$$
	with $K_\alpha = \frac{(1-\alpha)}{\alpha}$ for an appropriate $C_1=C_1(\eta,a,b)$ and $C_2=C_2(\eta,a,b)$.
	\end{lemma} 
\begin{proof}
That the function is a solution follows from direct computation.
We have to show that we can choose $C_1$ and $C_2$ in order to fulfil $w(0) = a$, $w(1) = b$ and for the function to be defined for $t\in(0,1)$.
We have
\[
\frac{1}{K_\alpha}\ln(\cos(C_1))+C_2=a
\]
and
\[
\frac{1}{K_\alpha}\ln\left(\cos\left(C_1+\frac{\eta\sqrt{1-\alpha}}{\alpha}\right)\right)+C_2=b.
\]
By subtracting the equations, after some computations, we get that
\begin{equation}
\label{eqC}
\frac{\cos\left(C_1+\frac{\eta\sqrt{1-\alpha}}{\alpha}\right)}{\cos(C_1)}
=e^{K_\alpha(b-a)}.
\end{equation}
Assume that $\eta$ is small enough such that $\frac{\eta\sqrt{1-\alpha}}{\alpha}<\pi/2$, then the LHS, that we call $F(C_1)$, is defined for $C_1\in(-\pi/2,0]$.
Observe that it is decreasing in $C_1$, $F(0)\leq 1$ and $\lim_{C_1\to -\pi/2^+}F(C_1)=+\infty$.
Then, since the RHS is greater than 1, we get that there exists a unique $C_1$ such that the equality holds.
Moreover, we also have $-\pi/2<C_1<C_1+\frac{\eta\sqrt{1-\alpha}}{\alpha}<\pi/2$ so the function is defined for every $t\in(0,1)$.
Finally, we conclude by taking
$$
C_2=a-\frac{\ln(\cos(C_1))}{K_\alpha}.
$$
\end{proof}

\begin{lemma}
\label{lem:convw}
Let $\alpha\in (0,1)$ and $b>a$.
The solution $w$ to \eqref{eqwitheta} converges uniformly to $v$ the solution to \eqref{Mikko} as $\eta\to 0$.
\end{lemma}
\begin{proof}
We have
\[
w(t) = \frac{1}{K_\alpha } \ln \left(
\frac{ \cos\left(C_\eta + \frac{\eta\sqrt{1-\alpha}}{\alpha} t\right)}
{\cos(C_\eta)}
 \right) + a.
\]
where $C_\eta:=C_1$ is given by \eqref{eqC}, that is
\begin{equation}
\label{eqCeta}
\frac{\cos\left(C_\eta+\frac{\eta\sqrt{1-\alpha}}{\alpha}\right)}{\cos(C_\eta)}
=e^{K_\alpha(b-a)}.
\end{equation}
We consider
\[
F_\eta(t)=\frac{ \cos\left(C_\eta + \frac{\eta\sqrt{1-\alpha}}{\alpha} t\right)}
{\cos(C_\eta)}
\]
and we want to show that it converges uniformly to 
\[
F(t)=1+(e^{(b-a)K_\alpha}-1)t
\]
as $\eta\to 0$.

Observe that the cosine function, $\cos (\cdot)$, is concave in $(-\pi/2,0)$. Hence, we have 
\[
\cos\left(C_\eta+\frac{\eta\sqrt{1-\alpha}}{\alpha}\right)
\leq
\cos\left(C_\eta\right)-\sin(C_\eta)\frac{\eta\sqrt{1-\alpha}}{\alpha}.
\]
Therefore, bounding $\sin (\cdot)$ by 1 we get
\[
e^{K_\alpha(b-a)}=\frac{\cos\left(C_\eta+\frac{\eta\sqrt{1-\alpha}}{\alpha}\right)}{\cos(C_\eta)}
\leq 
1+\frac{\eta\sqrt{1-\alpha}}{\alpha \cos\left(C_\eta\right)},
\]
and then we obtain
\[
 \cos\left(C_\eta\right)
\leq 
\frac{\eta\sqrt{1-\alpha}}{\alpha (e^{K_\alpha(b-a)}-1)}.
\]
From this, we conclude that $$\lim_{\eta\to 0}C_\eta=- \frac{\pi}{2}.$$

By the implicit function theorem we get
\[
\frac{\partial C_\eta}{\partial \eta}
=\frac{\sin(C_\eta +\eta \frac{\sqrt{1-\alpha}}{\alpha})\frac{\sqrt{1-\alpha}}{\alpha}}{-\sin(C_\eta +\eta \frac{\sqrt{1-\alpha}}{\alpha})+\sin(C_\eta) e^{K_\alpha(b-a)}}.
\]
Then
\[
\lim_{\eta\to 0}\frac{\partial C_\eta}{\partial \eta}
=\frac{\frac{\sqrt{1-\alpha}}{\alpha}}{e^{K_\alpha(b-a)}-1}.
\]
With this result, applying L'H\^{o}pital's rule to the quotient in the definition of $F_\eta$ we get $$\lim_{\eta\to 0} F_\eta(t)=F(t).$$ 
Finally, to conclude that the convergence is uniform we consider
\[
\frac{\partial F_\eta(t)}{\partial t}=
\frac{-\sin(C_\eta+\eta \frac{\sqrt{1-\alpha}}{\alpha}t) \eta \frac{\sqrt{1-\alpha}}{\alpha}}{\cos(C_\eta)}
\]
Since 
\[
\sin\left(C_\eta+\eta \frac{\sqrt{1-\alpha}}{\alpha}t\right)\to -1
\]
uniformly and by using again L'H\^{o}pital's rule we obtain that
\[
\lim_{\eta\to 0}\frac{\eta\frac{\sqrt{1-\alpha}}{\alpha}}{\cos(C_\eta)}=
e^{K_\alpha(b-a)}-1,
\]
so the derivatives converge uniformly to $e^{(b-a)K_\alpha}-1$ which is the derivative of $F$ and then we conclude the same convergence for the functions, $F_\eta \to F$
uniformly as $\eta \to 0$.
\end{proof}

	\subsection{Basic properties}
	
\label{sec:basic}
	
			%%%%%%%%%%%%%%%%%%%%%
We begin observing that the notions of convexity associated with 
	$L_\alpha$ are stronger as $\alpha$ increases. As we mentioned in the introduction, this is naturally consistent with the fact that convexity (the concept for $\alpha=1$) implies quasiconvexity
	($\alpha=0$).
		We also observe that the supremum of $\alpha$-convex functions is $\alpha$-convex.
		
	\begin{proposition}
	\label{prop:bigger-alpha-stronger}
	Let $\alpha > \hat{\alpha}$. If 
	$u$ is $\alpha$-convex, then $u$ is $\hat \alpha$-convex.
\end{proposition}
%%%%%%%%%%%%%%%%%%%%%

	\begin{proof}
Using the explicit expression found in Lemma~\ref{lemma-L-1d}  
we obtain that 
	when we fix the boundary values $a$, $b$, then $v_\alpha (t)$, the solution to 
	$$
	\left\{
	\begin{array}{l}
\displaystyle L_\alpha (v) (t) = \alpha v'' (t) + (1-\alpha) |v' (t)|^2 =0, \qquad t \in (0,1), \\
\displaystyle v(0) = a, \\ 
v(1) = b,
\end{array}
\right.
	$$
	is decreasing with respect to $\alpha$. 
	This implies the claim.
\end{proof}

	%%%%%%%%%%%%%%%%%%%%%%%%%%%%	
Alternatively the above proof could be obtained from the observation that $v_\alpha''\leq 0$ 
(this follows from the equation), and thus $v_\alpha$ is a subsolution to  $L_{\hat \alpha}v\ge 0$ for $\alpha > \hat{\alpha}$. Then, the result follows using the comparison principle from Lemma \ref{lem:1dcomp}.

%%%%%%%%%%%%%%%%%%%%%%%%%%%%%%%%%%	
	\begin{lemma} \label{lema.sup}
	Let $\{ u_n \}$ be a family of $\alpha$-convex functions. Then
	$$
	\overline{u} (z) = \sup \Big\{u_n(z) \Big\},
	$$
	is also $\alpha$-convex. 
\end{lemma}	
%%%%%%%%%%%%%%%%%%%%

\begin{proof}
	For each $u_n$ in the family, since $u_n$ is $\alpha$-convex, we have that for any pair of points
	$x$, $y$ in $\overline{\Omega}$ it holds that 
	$$
	u_n (tx+(1-t)y) \leq v(t) 
	$$
	with $v$ the solution to 
	$$
	\left\{
	\begin{array}{l}
\displaystyle L_\alpha v (t) =0, \qquad t \in (0,1), \\
\displaystyle v(0) = u_n(y), \\ v(1) = u_n(x).
\end{array}
\right.
	$$
	As $\overline{u} (x) \geq u_n(x)$ and $\overline{u} (y) \geq u_n(y)$ by the comparison 
	principle for $L_\alpha$ in Lemma~\ref{lem:1dcomp} in the interval $(0,1)$
	we have that 
	$$
	v(t) \leq \tilde{v} (t), \qquad t \in (0,1),
	$$
	with $ \tilde{v}$ the solution to
	$$
	\left\{
	\begin{array}{l}
\displaystyle L_\alpha \tilde{v} (t) =0, \qquad t \in (0,1), \\
\displaystyle \tilde{v}(0) = \overline{u} (y), \\ \tilde{v}(1) = \overline{u} (x).
\end{array}
\right.
	$$
	Therefore, we arrive to
	\begin{equation} \label{tttt}
	u_n(tx+(1-t)y) \leq \tilde{v}(t) .
	\end{equation}
	Taking supremum in the left hand side of \eqref{tttt} we obtain
	$$
	\overline{u} (tx+(1-t)y) \leq \tilde{v}(t) .
	$$
	This proves that $\overline{u}$ is $\alpha$-convex. \end{proof}
	
	\begin{remark}
	Also observe that the same proof shows that, for a general $1$-dimensional operator $L$,  the supremum of $L$-convex functions is $L$-convex
	if the operator $L$ has a comparison principle. 
	\end{remark}

	\subsection{First part of the proof of Theorem \ref{teo.1.intro}: Lipschitz continuity} 
	
	We show that for $\alpha \neq 0$, $\alpha$-convex functions are Lipschitz continuous inside the domain ${\Omega}$. This is the first part of Theorem \ref{teo.1.intro}, but as it is of independent interest,  we formulate the result separately.
	
	%%%%%%%%%%%%%%%%%%%%%%%%%%%%%%%%%%%%%
	
	\begin{theorem} \label{teo.cont}
	Let $\alpha \neq 0$. A bounded $\alpha$-convex function $u$ is Lipschitz continuous in $\Omega$.
	\end{theorem}
	
	\begin{proof}
	First we estimate $u$ from above.
Let $z_0\in\Omega$ and $r>0$ be such that $B_r(z_0)\subset\Omega$.
Since $u$ is $\alpha$-convex and bounded it turns out that for every point $\hat{z} \in \partial B_{r/2}(z_0)$ and every $t \in [0,1]$ we have
	$$
	u(t \hat{z} + (1-t) z_0 ) \leq v(t)
	$$ 
	with $v$ the solution to 
	$$
		\left\{
	\begin{array}{l}
\displaystyle L_\alpha v (t) =0, \qquad t \in (0,1), \\
\displaystyle v (0) = u(z_0), \\ 
v(1) =  \max u \geq u(\hat{z}).
\end{array}
\right.
	$$
	
Next we estimate $u$ from below. Recall that $v$ as above is defined for $t>-\delta$ where $\delta=\delta(|\max u-u(z_0)|,\alpha)>0$, see Remark~\ref{rem:delta}.
We have that
	$$
	u(t \hat{z} + (1-t) z_0 ) \geq v(t)
	$$ 
	for $t\in(-\delta,0)$, see Remark~\ref{rem:outside}.

We conclude that 
\[
v(-|z-z_0|)\leq u(z)\leq v(|z-z_0|)
\]
for every $z\in B_{R} (z_0)$ with $R=\min\{r/2,r\delta/4\}$. 
Since the derivative of $v$ in $[-\delta/2,1]$ is bounded we obtain that
the graph of $u$ is outside a cone with center at $(z_0,u(z_0))$, see Figure~\ref{fig:lip}.
We have proved that
$$
|u(z)- u(z_0)| \leq L |z-z_0|.
$$
\end{proof}	
	\begin{center}
\begin{figure}[h]
\centering
\begin{tikzpicture}[scale=1.5]
	\draw[fill=blue!30,draw=blue!30] plot[smooth, samples=100, domain=0.5:1] (\x,ln \x) -| (-1.5,ln 0.5) -- cycle;
	\draw[fill=blue!30,draw=blue!30] plot[smooth, samples=100, domain=-1:1] (\x,{ln(2-\x)}) -| (-1.5,ln 3) -- cycle;
	\draw[fill=blue!30,draw=blue!30] plot[smooth, samples=100, domain=1:3] (\x,ln \x) -| (3.5,0) -- cycle;
	\draw[fill=blue!30,draw=blue!30] plot[smooth, samples=100, domain=1:1.5] (\x,{ln(2-\x)}) -| (3.5,0) -- cycle;
		\draw plot [domain=0.5:3,samples=100] (\x,{ln(\x)});
		\draw plot [domain=-1:1.5,samples=100] (\x,{ln(2-\x)});
\draw [draw=red](0.6,ln 0.5) -- (1.4,-ln 0.5);
\draw [draw=red](1.4,ln 0.5) -- (0.6,-ln 0.5);
\fill[white,path fading=east] (-1.51,ln 0.5-0.01) rectangle (-1,ln 3+0.01);
\fill[white,path fading=west] (3,ln 0.5-0.01) rectangle (3.5+0.01,ln 3+0.01);
\end{tikzpicture}
	\caption{We have $v$ in black, the limit of the cone in red and the graph of $u$ must be inside the light blue region.}
	\label{fig:lip}
	\end{figure}
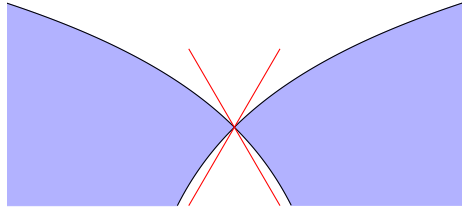
	\end{center}
	
	\begin{remark}
	For $\alpha=0$ the previous result does not hold. There are quasiconvex functions that are not continuous.
	For example, take
	$$
	u(x) = 
	\left\{
	\begin{array}{ll}
	0, \qquad & t \in [0,1/2), \\[4pt]
	1, \qquad & t \in [1/2,1], 
	\end{array}
	\right.
	$$
	that is quasiconvex in $[0,1]$.  
	\end{remark}
	
		\begin{remark}
	The obtained Lipschitz regularity is optimal since a cone
	$$
	u(z) = |z|
	$$
	is convex, and hence $\alpha$-convex for every $\alpha$, but it is not better than
	Lipschitz.
	
	Also notice that the optimal $\delta$ in the previous proofs goes to zero as
	$\alpha \to 0$. Therefore, we do not have a local Lipschitz 
	constant that is uniform in $\alpha$. This is natural (and expected) since for 
	$\alpha=0$ we have quasiconvex functions that can be discontinuous
	inside $\Omega$.
	\end{remark}

	\begin{remark} 
	A bounded $\alpha$-convex function $u$ may not be Lipschitz continuous in $\overline\Omega$.
	In fact, it may not be continuous up to the boundary, for example $u$ given by $u(x)=0$ in $[0,1)$ and $u(1)=1$ is not continuous and, since it is convex, it is $\alpha$-convex for every $\alpha\in[0,1]$.
	
	Also, it may happen that the Lipschitz constant deteriorates when approaching the boundary.
	For example, take $u(x)=-\sqrt{x}$ in $[0,1]$.
	\end{remark}
	
	\begin{remark} For a general operator $L$ the previous proof works provided $L$
	satisfies that  
 solutions to 
	$$
	\left\{
	\begin{array}{l}
\displaystyle L v (t) =0, \qquad t \in (0,1), \\
\displaystyle v(0) = a, \\ v(1) = b,
\end{array}
\right.
	$$ 
	are Lipschitz continuous and depend continuously on the boundary data $a,b$.
	\end{remark}
	
	\subsection{Second part of the proof of Theorem \ref{teo.1.intro}: Viscosity subsolution}

Now we will prove the second part of Theorem \ref{teo.1.intro}.
	We have already proven in Theorem \ref{teo.cont} that $\alpha$-convex functions are Lipschitz continuous.
	It remains for us to show that a function is $\alpha$-convex if and only if it is a subsolution to 
	\begin{align}
\label{eq:Lu}
			\mathcal{L}_\alpha  u (z) :&= 
	 \inf_{
	\substack{x,y \in \overline{\Omega}  \\
		z=t_0 x+ (1-t_0)y} }
		L_\alpha (u) (t_0)  \nonumber
		\\
		& =\inf_{|v|=1} \! \! 
		\Big\{\alpha \langle D^2 u(z) v, v \rangle \! + \! (1-\alpha)  |\langle D u(z) ,v \rangle |^2 \Big\} \! = 0, \qquad  z \in \Omega,
\end{align}
in terms of one dimensional test functions as in Definition \ref{defi-visc}.

	Actually we notice that we have two natural notions of viscosity solution to 
this equation.
		In Definition \ref{defi-visc.N} (that corresponds to what is usual in the viscosity theory), we test with $N-$dimensional functions 
		$\phi:\Omega \to \mathbb{R}$ that touches
		$u$ from above at $z\in \Omega$ and
		we ask for 
		$$
\alpha \langle D^2 \phi (z) v, v \rangle \! + \! (1-\alpha)  |\langle D \phi (z) ,v \rangle |^2  \geq 0
		$$
		for any direction $v \in\mathbb{S}^{N-1}$. 	We also assume the reverse inequality when the test function touches $u$ from below at $z$ in the 
		$N-$dimensional set $\Omega$. 
		
		The alternative definition (Definition \ref{defi-visc} below) changes the definition of a subsolution: we first take a direction 
		$v \in\mathbb{S}^{N-1}$ and then
		a 
		$1-$dimensional test function 
		$\phi$ that touches $u$ from above at $z$ 
		in the $1-$dimensional set  
		$\{ z+ t v \} \cap \Omega$. 
		Notice that now $\phi$ needs only to be defined in the $1-$dimensional set and not in the whole 
		$\Omega$. Here
		we ask for 
		$$
	\alpha \phi '' (z)  \! + \! (1-\alpha)  | \phi ' (z) |^2  \geq 0.
		$$

		Observe that if $u$ is a viscosity solution according to this second definition then it is a solution according to the first one.
		This is due to the fact that when an $N-$dimensional test function $\phi$ touches 
		$u$ from above/below at $z$ in $\Omega$, then the restriction
		of $\phi$ to any line $\{z+t v\}$, touches $u$ from above/below at $z$ 
		inside the line.
		The converse is more delicate since given a $1-$dimensional test function that touches $u$ 
		in a segment there is no immediate way of obtaining an $N-$dimensional test function $\psi$
		that touches $u$ in $\Omega$ and such that the restriction of $\psi$ to the segment is $\phi$. However, both notions of solution are equivalent as later recorded in Remark~\ref{rem:conincide} since 
		we can prove uniqueness in Corollary~\ref{teo.uni} in the sense of Definition \ref{defi-visc.N}.
		
	\begin{definition}[Viscosity solutions using $1$-d tests for
	subsolutions] \label{defi-visc} \
		An upper semicontinuous function $u\colon\mathbb{R}^N\to\mathbb{R}$  is a 
		\textit{viscosity subsolution}  of \eqref{eq:Lu}
				if  for any $z\in \Omega,$ any pair 
		$x,y\in\Omega,$ such that $z=t_0 x+ (1-t_0)y$ for some $t_0 \in (0,1)$, 
		 any test function $\phi\in C^{2}(\mathbb{R})$ 
		such that 
		$$
		\phi(t_0)= u(z)\text{ and }\phi(t)\ge u(tx + (1-t)y) \quad t\in (0,1),
		$$
		we have
		\[
			L_\alpha \phi (t_0 ) \ge 0 .
		\]
		
		A lower semicontinuous function $u:\Omega \to\mathbb{R}$  is a 
		\textit{viscosity supersolution}  of 
		\eqref{eq:Lu} 
		if for any 
		$z\in \Omega,$ 
		any test function $\phi\in C^{2}(\Omega)$ 
		such that $$\phi(z)= u(z) \text{ and }\phi(y)\le u(y) \text{ in } \Omega,$$ 
we have		
		\[
			\mathcal{L}_{\alpha} \phi (z)\leq 0 .
		\]

		Finally, we say that $u$ is a viscosity solution of \eqref{eq:Lu} when it is both a 
		viscosity subsolution and a viscosity supersolution of	\eqref{eq:Lu}.
	\end{definition}

	A function $u:\Omega \to \mathbb{R}$ is a viscosity subsolution 
		of \eqref{eq:Lu} if only if for any $z\in\Omega$ and any $x,y\in \overline{\Omega}$
		with  $z=t_0 x+ (1-t_0)y$ for some $t_0 \in (0,1)$
		we have that the function $w(t)=u(tx+(1-t)y)$, $t \in (0,1)$, is a viscosity subsolution of the 
		one-dimensional problem
		\[
			L_\alpha w(t)=0 \quad \text{ for } t \in (0,1).
		\]
		
\begin{remark}
\label{rem:max}
In the introduction we claimed that $v(t) = \max \left\{ u(x),u(y)\right\}$ is the viscosity solution to the equation $|v' (t)|=0$ 
	in the interval $(0,1)$ with boundary conditions $v(1) = u(x) $ and $v(0) = u(y)$.
	To be more precise, if $u(x)>u(y)$ the solution is given by $v(t) = \max \left\{ u(x),u(y)\right\}$ for $t\in(0,1]$ and $v(0)=u(y)$.

This function verifies the boundary condition in the following sense.
For any test function $\phi$ that touches the upper envelope of $u$ from above we have $\max\{|\phi'|, g-\phi\}\geq 0$. 
And for any test function $\phi$ that touches the lower envelope of $u$ from below we have $\min\{|\phi'|, g-\phi\}\leq 0$. 
\end{remark}

	\begin{proof}[Proof of Theorem \ref{teo.1.intro}]
	We already showed in Theorem \ref{teo.cont} that $\alpha$-convex functions are Lipschitz continuous.
	To show that a function is $\alpha$-convex if and only if it is a subsolution to (\ref{eq:Lu}) assume first that $u$ is $\alpha$-convex. To show that 
	\begin{equation} \label{convex-envelope-X.99.88}
	\mathcal{L}_\alpha  u (z) := 
	 \inf_{
	\substack{x,y \in \overline{\Omega}  \\
		z=t_0 x+ (1-t_0)y} } L_\alpha (u) (t_0) \geq 0, \qquad z \in \Omega,
		\end{equation}
	in the viscosity sense we argue by contradiction.
	Hence, assume that there exists a point
	$z_0 \in \Omega$, two other points $x_0, y_0 \in \overline{\Omega}$
	with $z_0=t_0 x_0+ (1-t_0)y_0$
	 and a smooth $1$-dimensional
	test function $\varphi$ that touches $t \mapsto u (t x_0 + (1-t)y_0)$ from above in the interval $(0,1)$ at $t_0$ with
	$$
	L_\alpha \varphi (t_0) < 0. 
	$$
	As $\varphi$ is smooth we have that there exists $\delta >0$ such that
	$$
	L_\alpha \varphi (t) < 0
	$$
	for every $t \in (t_0 -\delta, t_0 + \delta)$.

	Now, consider the segment $t x_0 + (1-t) y_0$ with $t \in (t_0-\delta, t_0+\delta)$ and $v$ the solution to
	$$
	\left\{
	\begin{array}{l}
\displaystyle L_\alpha v (t) =0, \qquad t \in (0,1), \\
\displaystyle v(0) = u((t_0-\delta)x_0 + (1-(t_0-\delta))y_0), \\ 
v(1) = u((t_0+\delta)x_0 + (1-(t_0+\delta))y_0).
\end{array}
\right.
	$$
	
	Since $L_\alpha$ behaves well under scaling (see Remark~\ref{rem:scaling})
we know that 
	$$
	w(t) = v \Big(  \frac{1}{ 2 \delta} ( t- (t_0-\delta ) )  \Big)
	$$ 
	solves the equation 
	$L_\alpha w =0$ for $t \in (t_0-\delta, t_0+\delta)$ and $\varphi$ is a strict supersolution
	to the same equation in that interval. 
	Since $L_\alpha$ has a strict comparison principle (see Lemma \ref{lem:1dcomp}),  and 
	we have $$u((t_0-\delta)x_0 - (1-(t_0-\delta) y_0) \leq \varphi (t_0- \delta)$$ and $$u((t_0+\delta)x_0 - (1-(t_0+\delta) y_0) \leq \varphi (t_0+\delta)$$
	 (recall that $\varphi$ touches $u$ from above at $t_0x_0 + (1-t_0)y_0$
	on the line $tx_0 +(1- t) y_0$) we obtain
	$$
	w(t) < \varphi (t), \qquad \text{ for all } t \in  (t_0-\delta, t_0+\delta).
	$$
	In particular, it holds that 
	$$
	 v(1/2) = w(t_0) < \varphi (t_0) = u (x_0)
	$$
	a contradiction with the fact that $u$ is $\alpha$-convex.

	Next we show that the property of being a subsolution implies the $\alpha$-convexity.  
	We assume that $u$ solves 
	\begin{equation} \label{convex-envelope-X.556}
	\mathcal{L}_\alpha  u (z) := 
	 \inf_{
	\substack{x,y \in \overline{\Omega}  \\
		z=t_0 x+ (1-t_0)y} } L_\alpha (u) (t_0) \geq 0, \qquad z \in \Omega.
		\end{equation}
		To show that $u$ is $\alpha$-convex we argue again by contradiction. 
		Assume that there are two points $x_0, y_0 \in \Omega$ and $z_0 = t_0 x_0 + (1-t_0) y_0$
		such that
		$$
		u(z_0) > v(t_0)
		$$
		with $v$ the solution to 
		$$
		\left\{
	\begin{array}{l}
\displaystyle L_\alpha v (t) =0, \qquad t \in (0,1), \\
\displaystyle v(0) = u(y_0), \\ 
v(1) = u(x_0).
\end{array}
\right.
	$$
	Consider $w$ the solution to
	$$
	\left\{
	\begin{array}{l}
\displaystyle L_\alpha w (t) = - \eta^2, \qquad t \in (0,1), \\
\displaystyle w(0) = u(y_0), \\ w(1) = u(x_0),
\end{array}
\right.
	$$
	with a small $\eta >0$, see Lemma~\ref{lem:1d-eta}.
	By Lemma~\ref{lem:convw}, the solution $w$ converges uniformly to $v$.
	Hence, if $\eta$ is small we still have that  
	$$
		u(z_0) > w(t_0).
		$$
	Now take $k>0$ such that
	$w+k$ touches $u$ from above at some point $t_1$ in the segment
	 $t x_0 + (1-t) y_0$ (notice that at the extreme points $t=0$ and $t=1$ we have 
	 $w+k > u$ for any $k>0$, hence $t_1 \in (0,1)$. At this point $t_1$ we have
	 $L_\alpha w (t_1) = - \eta^2$ that contradicts the fact that $u$ is a solution to \eqref{convex-envelope-X.556}.
	\end{proof}
	
%%%%%%%%%%%%%%%%%%%%%%%%%%%%%%%%%%
		
		\subsection{Supporting $\alpha$-hyperplane}
		
 Next we define a concept that will be analogous to a supporting hyperplane for convex functions.
	We look at a hyperplane as a function that is affine in one direction and does not depend on the orthogonal coordinates. 
	Hence, in the context of $\alpha$-convex functions, the following is a natural generalization.

\begin{definition}[$\alpha$-hyperplane]
\label{def:alpha-hyperplane}
	An {\it $\alpha$-hyperplane} passing through $z_0$  with direction $\nu$ is defined as a function of the form
	$$
	\pi_\alpha (z) = v(\langle z-z_0, \nu \rangle)
	$$
	with $v$ a solution to
	$$
	L_\alpha (v)(t) = \alpha v''(t) + (1-\alpha) |v' (t)|^2 =0
	$$
	for $t$ in some interval of the form $I = (-\delta, + \infty)$. 
\end{definition}
%%%%%%%%%%

\begin{remark} \label{1-d--2-d}
	 If we have a function that depends only on the first coordinate, then it is 
	$\alpha$-convex in a multidimensional domain if and only if
	it is $\alpha$-convex in the $1$-dimensional projection of the domain into the first coordinate.
	Indeed, if we have $$u(z_1,z_2, \dots,z_n) = f(z_1)$$ then for every $\alpha \neq 0$:
	$
	u$  is $\alpha$-convex in $\Omega \subset \mathbb{R}^n$  can be equivalently stated by saying that  $f$ is $\alpha$-convex 
	in $I\subset \mathbb{R}$. 
	In fact, we have that
	$$
	\begin{array}{l}
	\displaystyle 
	 \mathcal{L}_\alpha  u (z) =  \inf_{v=(v_1,v_2,\dots,v_n) \colon |v|=1} 
		\alpha \langle D^2 u(z) v, v \rangle + \alpha  |\langle D u(x), v \rangle |^2  \\[7pt]
		\displaystyle \qquad \qquad
		=  \inf_{v=(v_1,v_2,\dots,v_n) \colon |v|=1} \alpha f'' (z_1) v_1^2 + (1-\alpha) (f' (z_1))^2  v_1^2,
		\end{array}
		$$
		and hence we have
		$$
		 \mathcal{L}_\alpha  u (z)  \geq 0 \text{ is equivalent to } L_\alpha f (z_1) \geq 0.
		$$
		In particular, an $\alpha$-hyperplane is $\alpha$-convex.
	\end{remark}

Observe that an $\alpha$-hyperplane can be determined by choosing a base point $z_0$, 
the value of $\pi_\alpha(z_0)$, the direction $\nu$ and $v'(0)$, or equivalently by  $z_0$, $\pi_\alpha(z_0)$ and $ D  \pi_\alpha(z_0)$.
Now, notice that, recalling Remark~\ref{rem:delta}, that this function is defined in a half-space that contains $z_0$ in its interior.

Next we show that an $\alpha$-convex function has supporting $\alpha$-hyperplanes
at every point inside $\Omega$
(that is, given a point $z\in \Omega$, there exists an $\alpha$-hyperplane touching the $\alpha$-convex
function from below at $z$). The idea is to take the comparison function $v$ and the two suitable points in the definition of the $\alpha$-convexity, let the second point approach the first one in Lemma \ref{Lem:touchingFromBelow}, and then show that $v$ tends to the desired function. To show that $v$ indeed is the function that defines the $\alpha$-hyperplane that twe look for, we first assume differentiability of $u$ in Lemma~\ref{lemma.touching.below.diff} and then complete the proof by an approximation argument.

	\begin{theorem}[Supporting $\alpha$-hyperplane] \label{chess}
	For $\alpha \neq 0$ let $u$ be a bounded $\alpha$-convex function in $\Om$. 
	Then, given $z_0 \in \Omega$, there exists an {\it $\alpha$-hyperplane} 
	$\pi_\alpha$ that touches $u$ from below at $z_0$.
	\end{theorem}

In the case of quasiconvex functions we have that the level sets are convex. 
So given $z_0$ we can consider the set $\{x:u(x)<u(z_0)\}$ and we get that there exists $\nu$ such that 
\[
\{x:u(x)<u(z_0)\}\subset \{x:\langle x-z_0, \nu \rangle<0\}.
\]
Therefore if we consider $\pi$ the constant function $u(z_0)$ we have that $u\geq \pi$ in the half-space $\{x:\langle x-z_0, \nu \rangle\geq 0\}$.
We can see the existence of supporting $\alpha$-hyperplanes for $\alpha$-convex functions interpolating between what can be obtained for convex and quasiconvex functions.

We begin our way to the proof of Theorem~\ref{chess} with the following lemma.
\begin{lemma}
\label{Lem:touchingFromBelow}
Let $u:\Omega\to\R$ be an $\alpha$-convex function, $z_0\in\Omega$ and $\nu\in\R^N$.
Then there exists $v$ such that $L_\alpha v (t) =0$,  $v(0)=u(z_0)$ and
\[
v(t)\leq u(z_0+\nu t)
\]
for every $t$ such that both functions are defined.
\end{lemma}
\begin{proof}
We consider $v_n$ the solution to the equation $L_\alpha v_n=0$ with $v_n(0)=u(z_0)$ and $v_n(1/n)=u(z_0+\nu /n)$.
Since $u$ is Lipschitz, we get that $v_n'(0)$ is uniformly bounded by looking at the explicit formula.
Therefore, there exists $\delta>0$ such that $v_n$ is defined in $[-\delta,\delta]$ for every $n$.
Also, we can take a subsequence such that it converges uniformly, that is, there exists $v$ such that $v_{n_k}\to v$ uniformly in $[-\delta,\delta]$. By stability of viscosity solutions, $v$ is a solution to the equation and for the limit it holds that  $v(t)\leq u(z_0+\nu t)$ recalling  Remark~\ref{rem:outside}.
\end{proof}
%%%%%%%%%%%%%%%%%%%%%%%%%%%%%%%%%%%%%%

	Now, let us show that at every point where an $\alpha$-convex function $u$ is differentiable there is a supporting $\alpha$-hyperplane that touches $u$ from below (and next we will extend this property to every point in $\Omega$ by an approximation argument). 
	
	\begin{lemma} \label{lemma.touching.below.diff}
	For $\alpha \neq 0$ let $u$ be a bounded $\alpha$-convex function. Assume that
	$u$ is differentiable at $z_0 \in \Omega$, then there exists an {\it $\alpha$-hyperplane} 
	$\pi_\alpha$ that touches $u$ from below at $z_0$. 
\end{lemma}	

\begin{proof} First, assume that $ D  u (z_0) \neq 0$.
Let $\nu = \frac{ D  u}{| D  u|} (z_0)$.
By Lemma~\ref{Lem:touchingFromBelow} we know that there is a
solution to $L_\alpha v(t) =0$ in $(-\delta,+\infty)$ that touches $u$ from below in the line
$z_0 + t \nu$.
 We have $u(z_0) = v(0)$ and $ u (z_0 + t \nu) 
	\geq v(t)$ for $t\in (- \delta,+\infty)$ such that $u$ is defined. Take as $\pi_\alpha (z)$ the $\alpha$-hyperplane 
	associated with this solution and the direction $\nu$, that is, 
	$$
	\pi_\alpha (z) = v(\langle z-z_0, \nu \rangle).
	$$
We have that $$\pi_\alpha (z_0) = v(z_0) = u(z_0).$$ 
Moreover, we have that $t \mapsto (u-\pi_\alpha) (z_0 + t \nu)$
attains a minimum at $t=0$. Therefore,
since $u$ is differentiable at $z_0$ 
we have 
$$
\frac{\partial }{\partial t} (u-\pi_\alpha) (z_0 + t \nu) |_{t=0} = 0.
$$
Hence, since  $\nu$ is the direction of the gradient of $u$  we get
$$
\frac{\partial }{\partial t} \pi_\alpha (z_0 + t \nu)|_{t=0}
= \langle  D  \pi_\alpha (z_0), \nu \rangle = \langle  D  u (z_0), \nu \rangle = | D  u (z_0)|. 
$$
Therefore, we have
$$
 D  \pi_\alpha (z_0) = | D  u (z_0)| \nu =  D  u(z_0). 
$$

Let us show that $$\pi_\alpha (z) \leq u(z)$$ for every $z \in \Omega$ with $-\delta < \langle z-z_0, \nu \rangle $.
To this end take any such $z$ and consider $w$ the direction of $z-z_0$. 
Since we have $ D  \pi_\alpha (z_0) =  D  u(z_0) $ if we compute derivatives 
in the direction $w$ we have 
\begin{equation}
\label{derivatives}
\frac{\partial }{\partial t} u (z_0 + t w)|_{t=0} = \langle  D  u(z_0), w \rangle =
\langle  D  \pi_\alpha (z_0), w \rangle
=
\frac{\partial }{\partial t} \pi_\alpha (z_0 + t w)|_{t=0} . 
\end{equation}
Hence, the functions
$$
t \mapsto u (z_0 + t w) \qquad \mbox{ and }\qquad t\mapsto \pi_\alpha (z_0 + t w)
$$ 
are a subsolution and a solution to $L_\alpha v =0$ that coincide at $t=0$ and has the same derivative at $t=0$.
We claim that $u (z_0 + t w) < \pi_\alpha (z_0 + t w)$.
In fact, for the sake of contradiction, suppose that $u (z_0 + t_0 w) < \pi_\alpha (z_0 + t_0 w)$ and consider 
$$
\left\{
\begin{array}{l}
\displaystyle L_\alpha \tilde v (t) =0, \qquad t \in (0,t_0), \\
\displaystyle \tilde v(0) = u(z_0), \\ 
\tilde v(t_0) = u(z_0 + t_0 w).
\end{array}
\right.
	$$
Since $t\mapsto \pi_\alpha (z_0 + t w)$ solves the equation, $u (z_0) = \pi_\alpha (z_0)$ and $u (z_0 + t_0 w) < \pi_\alpha (z_0 + t_0 w)$, from Lemma~\ref{lem:1dcomp} we get $\tilde v(t)\leq \pi_\alpha (z_0 + t w)$ and therefore
$
\tilde v' (0) \leq
\frac{\partial }{\partial t} \pi_\alpha (z_0 + t w)|_{t=0}.
$
If $\tilde v' (0) = \frac{\partial }{\partial t} \pi_\alpha (z_0 + t w)|_{t=0}$ by the uniqueness of solution to the 1-dimensional problem we get that the functions coincide everywhere, which does not hold, so we conclude that
$$
\tilde v' (0) <
\frac{\partial }{\partial t} \pi_\alpha (z_0 + t w)|_{t=0}.
$$
But since $u$ is $\alpha$-convex, it is is a subsolution to the PDE. Then, from a comparison argument we get $u\leq \tilde v$ and therefore
$$
\frac{\partial }{\partial t} u (z_0 + t w)|_{t=0} \leq \tilde v' (0)
$$
which contradicts \eqref{derivatives}.
Hence we have shown that $u (z_0 + t w) \geq \pi_\alpha (z_0 + t w)$ and we conclude that $u(z) \geq \pi_\alpha (z)$.
\end{proof}

	Now, let us remove the assumption that $u$ is differentiable at the point by 
	using an approximation argument.

	\begin{proof}[Proof of Theorem~\ref{chess}] Take the supremal convolution of $u$,
		$$
		u_\gamma (x) = \sup_{y \in \Omega } \Big(  u(y) - \frac{1}{2\gamma} 
		|x-y|^2 \Big).
		$$
		Since $u$ is continuous by Theorem~\ref{teo.1.intro}, these functions $u_\gamma$ verify that
		$$
		\lim_{\gamma \to 0} u_\gamma = u 
		$$
		uniformly in $\overline{\Omega}$. Moreover, $u_\gamma$ is semiconvex 
		(and hence differentiable a.e. in $\Omega$) and, since $u$ is a subsolution to 
		 \begin{equation} \label{form-alter.99}
	 \begin{array}{l}
\displaystyle \mathcal{L}_\alpha  u (z) := \inf_{
	\substack{x,y \in \overline{\Omega}  \\
		z=t_0 x+ (1-t_0)y} } L_\alpha (u) (t_0)  \\[6pt]
\qquad \qquad \displaystyle =  \inf_{|v|=1} 
		\alpha \langle D^2 u(z) v, v \rangle + (1-\alpha)  |\langle D u(z), v \rangle |^2 \geq 0,
		\end{array}
		\end{equation}
		then $u_\gamma$ is also a subsolution in a slightly smaller domain,
		see \cite{CIL}. Therefore, $u_\gamma$ is also $\alpha$-convex. 
		
		Choose $z_\gamma$  a sequence of points such that $z_\gamma$ 
		is a differentiability point of $u_\gamma$ and 
		$$
		\lim_{\gamma \to 0} z_\gamma = z_0. 
		$$
		Using our previous result we have that there exist directions $\nu_\gamma$ and 
		solutions $v_\gamma$ of $L_\alpha (v_\gamma) (t) =0$ in $(-\delta, + \infty)$
		such that the corresponding 
		$\alpha-$hyperplanes given by 
		$$
		\pi_{\alpha, \gamma} (z) = v_\gamma (\langle z-z_\gamma, \nu_\gamma \rangle)
	$$
		touch $u_\gamma$ from below at $z_\gamma$, that is, 
		\begin{equation} \label{king}
		\pi_{\alpha, \gamma} (z_\gamma ) = u_\gamma (z_\gamma) 
		\end{equation}
		and 
		\begin{equation} \label{queen}
		\pi_{\alpha, \gamma} (z ) \leq u_\gamma (z) , \qquad \mbox{for } \langle z-z_\gamma, \nu_\gamma \rangle 
		> -\delta.
		\end{equation}
		Now, from compactness of $\{\nu : |\nu|=1\}$ we can extract a subsequence  $\gamma_j \to 0$
		such that $\nu_{\gamma_j} \to \nu_0$.
		Moreover, since $u$ is bounded so is $u_\gamma$ and $v_\gamma$, then	
		from the explicit formula for the solutions
		to $L_\alpha v_\gamma =0$, we can also extract a subsequence such that 
		$$
		v_{\gamma_j} \to v_0
		$$ 
		locally uniformly in $(-\delta, +\infty)$ with $v_0$ also a solution to $L_\alpha v_0=0$ by stability of viscosity solutions.
		Therefore, we have that 
		$$
		\pi_{\alpha, \gamma} (z) = v_\gamma (\langle z-z_\gamma, \nu_\gamma \rangle)
		\to \pi_{\alpha,0} (z) = v_0 (\langle z-z_0, \nu_0 \rangle)
	$$
	locally uniformly in $\{z : \langle z-z_\gamma, \nu_0 \rangle 
		> -\delta \}$.
		
		Now, we can pass to the limit in \eqref{king} and \eqref{queen} and obtain 
$$
		\pi_{\alpha, 0} (z_0 ) =  v_0 (0) = u (z_0) 
		$$
		and 
		$$
		\pi_{\alpha, 0} (z ) = v_0 (\langle z-z_0, \nu_0 \rangle)\leq u 
		(z) , \qquad \mbox{for } \langle z-z_0, \nu_0 \rangle 
		> -\delta,
		$$
		that is, $\pi_{\alpha,0}$ is an $\alpha$-hyperplane that touches $u$ from below at $z_0$.
	\end{proof}

		\subsection{Other properties}
		
	Convexity and quasiconvexity differ in how 
	they behave under monotone transformations.
	In fact, whether
	or not a function is convex depends on the numbers which the function assigns
	to its level sets, not just on the shape of these level sets. Then, a
	monotone transformation of a convex function need not be convex. That is,
	if $u$ is convex and $f:\mathbb{R} \to \mathbb{R}$ is increasing then 
	$f\circ u$ may fail to be convex (however, it is convex when $f$ is also convex). 
	For instance, $u(x)=x^2$ is convex and 
	$f(x)=\arctan(x)$ is increasing but $f\circ u(x)$ is not convex.  
	However, the weaker notion of  quasiconvexity keeps this property under monotonic
	transformations. Moreover, every monotonic transformation of a convex function
	is quasiconvex, although it is not true that every quasiconvex function
	can be written as a monotonic transformation of a convex function. 
	
	Concerning the composition of an $\alpha$-convex function with a smooth
	function $f:\mathbb{R} \to \mathbb{R}$ we have the following result.
	
	\begin{proposition} \label{prop.compos} For $0<\alpha<1$,
	let $u:\overline{\Omega} \to \mathbb{R}$ be an $\alpha$-convex function
	and $f:\mathbb{R} \to \mathbb{R}$ be such that
	$$
	f'(s) \geq 0, \qquad \mbox{and} \qquad \alpha f''(s) + (1-\alpha) [ (f'(s))^2 - f'(s) ] \geq 0.
	$$
	Then, the composition
	$$
	f \circ u :\overline{\Omega} \to \mathbb{R}
	$$
	is also $\alpha$-convex. 
	\end{proposition}
	
	\begin{proof}
	Since $u$ is $\alpha$-convex we have that for every $x,y \in \overline{\Omega}$,
	$$
	u(tx+(1-t)y) \leq v(t),\qquad t \in (0,1),
	$$
	with $v$ the solution to 
	$$
		\left\{
	\begin{array}{l}
\displaystyle L_\alpha v (t) =0, \qquad t \in (0,1), \\
\displaystyle v(0) = u(y_0), \\ 
v(1) = u(x_0).
\end{array}
\right.
	$$
	Now, as $f' (s) \geq 0$, we get that
	$$
	f(u)(tx+(1-t)y) \leq  f(v)(t),\qquad t \in (0,1).
	$$
	Take 
	$$
	\tilde{v} (t) = f(v)(t).
	$$
	We have that
	$$
	(1-\alpha ) |\tilde{v}' (t) |^2 = (1-\alpha ) ( f' (v(t)) )^2 |v'(t)|^2
	$$
	and
	$$
	\alpha \tilde{v}'' (t) = \alpha f''(v(t) |v'(t)|^2 + \alpha f'(v(t)) v'' (t).
	$$
	Therefore, we get
	$$
	\begin{array}{l}
	\displaystyle \alpha \tilde{v}'' (t)  + (1-\alpha)  |\tilde{v}' (t) |^2
	\\[7pt]
	\displaystyle = \alpha f''(v(t)) |v'(t)|^2 + (1-\alpha) ( f' (v(t)) )^2 |v'(t)|^2 -
	(1-\alpha) f' (v(t))  |v'(t)|^2 \\[7pt]
	=|v'(t)|^2 \left[ 
	 \alpha f''(v(t)) + (1-\alpha) ( f' (v(t)) )^2  - (1-\alpha) f' (v(t))
	\right] \geq 0,
	\end{array}
	$$
	provided $f$ verifies
	$$
	\alpha f''(s) + (1-\alpha) [ (f'(s))^2 - f'(s) ] \geq 0.
	$$
	Then, $\tilde{v}$ is a subsolution to 
	\begin{equation} \label{sauna}
		\left\{
	\begin{array}{l}
\displaystyle L_\alpha \hat{v} (t) =0, \qquad t \in (0,1), \\
\displaystyle \hat{v}(0) = f(u)(y_0), \\ 
\hat{v}(1) = f(u)(x_0),
\end{array}
\right.
	\end{equation}
	and hence we get
	$$
	f(u)(tx+(1-t)y) \leq  f(v)(t) \leq \hat{v} (t)\qquad t \in (0,1),
	$$
with $\hat{v}$ the solution to \eqref{sauna}. This shows that $f\circ u$ is
$\alpha$-convex.
	\end{proof}
	
	We also have the desirable property that an $\alpha$-convex function attains its minimum
	inside ${\Omega}$ when the $\inf_\Omega u$ can be localized inside some compact subset
	(as for usual convex functions some coercivity is needed).
	
	%%%%%%%%%%%%%%%%%%%%%%%%%%
	\begin{proposition} \label{alc.mim} Let $u:\overline{\Omega} \to \mathbb{R}$
	be an $\alpha$-convex function such that there exists a compact $K \subset \Omega$ such that
	$$\inf_K u = \inf_{\Omega} u.$$ 
	Then $u$
	attains its minimum
	in ${\Omega}$.
	If, moreover, $u$ is strictly $\alpha$-convex
	(the inequality in Definition \ref{L-convex-defi} is strict), then 
	there is a unique minimum point in $\overline{\Omega}$ unless $u$ is constant.
	\end{proposition}
	%%%%%%%%%%%%%%%%%%%%%%%
	\begin{proof}
	Since there exists a compact $K \subset \Omega$ such that
	$$\inf_K u = \inf_{\Omega} u,$$ 
 and $u$ is continuous 
 the existence of a minimum point in $K \subset \Omega$ follows.
	
	Now, assume that $u$ is strictly $\alpha$-convex
	and that there are two different points $x,y \in {\Omega}$ with $u(x) = u(y)=A$ that attain the minimum value $A$.
	Then, since $u$ is strictly $\alpha$-convex we have
	\begin{equation} \label{uuu}
	u(tx+(1-t)y) < v(t) \equiv A.
	\end{equation}
	Here we used that the solution to 
	$$
	\left\{
	\begin{array}{l}
\displaystyle L_\alpha v (t) =0, \qquad t \in (0,1), \\
\displaystyle v(0) = A, \\ v(1) = A,
\end{array}
\right.
	$$
	 is just the constant
	 $$
	 v (t) \equiv A.
	 $$
	Now (\ref{uuu}) contradicts the fact that $A$ was the minimum.
	\end{proof}
	
	\begin{remark} The condition that there exists a compact $K \subset \Omega$ such that
	$$\inf_K u = \inf_{\Omega} u,$$ 
	is needed since the function 
	$$
	u(t) = \left\{
	\begin{array}{ll}
	1 \qquad & t=0, \\[7pt]
	t \qquad & t \in (0,1],
	\end{array}
	\right.
	$$
	is convex in $[0,1]$ and also $\alpha$-convex for all $\alpha$ but it does not attain 
	a minimum in $[0,1]$.
	\end{remark}

	\section{The $\alpha$-convex envelope} \label{sect-envelope}
	
	In this section we deal with the $\alpha$-convex envelope of a continuous boundary datum $g$ in a strictly convex domain $\Omega \subset \mathbb{R}^N$ (with $N\geq 2$) as in Definition \ref{def:convex-env}:
\begin{equation} \label{convex-envelope-X.99}
	u^* (z) = \sup \Big\{w(z) : w \mbox{ is $\alpha$-convex and verifies } w|_{\partial \Omega} \leq g \Big\}.
		\end{equation}
		 Recall that $\Omega$ is strictly convex if $xt+y(1-t)\in \Omega$ for every $x,y\in \overline \Omega$ and $t\in (0,1)$.
We begin by proving that these functions are continuous (up to the boundary).
Then we obtain a comparison principle for $\mathcal{L}_\alpha$ (with viscosity solutions understood in the classical sense, testing with $N$-dimensional test functions).
Finally, we characterize the $\alpha$-convex envelope as the unique solution to the equation $\mathcal{L}_\alpha u=0$ inside $\Omega$ with boundary datum $g$ on $\partial \Omega$.

\subsection{Proof of Theorem~\ref{thm:cont}: equicontinuity of the $\alpha$-convex envelope.}

We begin with two preliminary lemmas.

\begin{lemma}
\label{starstar}
Let $V\subset\Omega$ be strictly convex, and $g:\partial \Omega \to \mathbb{R}$ a continuous function.
Let $u^*$ be the $\alpha$-convex envelope of $g$, then  the $\alpha$-convex envelope of $u^*|_{\partial V}$ in $V$ is just $u^*|_{V}$.
\end{lemma}

\begin{proof}
We consider 
\[
	u^{**} (z) = \sup \Big\{w(z) : w:V\to\R \mbox{ is $\alpha-$convex and verifies } w|_{\partial V} \leq u^*|_{\partial V} \Big\}.
\]
Since $u^*|_{V}$ is $\alpha$-convex we have $u^{**}\geq u^*$ in $V$.

We consider $u$ defined in $\Omega$ given by
\[
u(x)=
\begin{cases}
u^{**}(x) & \text{ if } x\in V\\
u^{*}(x) & \text{ if } x\not\in V.
\end{cases}
\]
We claim that this function is $\alpha$-convex.
Then $u\leq u^*$ and we conclude $u^{**}\leq u^*$ in $V$ as desired.

It remains to prove the claim.
Let $x,y\in\Omega$ and $t \in (0,1)$, we have to show that	$u(tx+(1-t)y) \leq v(t)$ where $v$ is the solution to 
	$$
	\left\{
	\begin{array}{l}
\displaystyle L v (t) =0, \qquad t \in (0,1), \\
\displaystyle v(0) = u(y), \\ 
v(1) = u(x).
\end{array}
\right.
	$$
Since $u^{**}\ge u^{*}$, we always have 
\begin{align*}
v(0)\ge u^{*}(y),\quad v(1)\ge u^{*}(x),
\end{align*}
and thus $v(t)\ge u^{*}(tx+(1-t)y)$ always.
Since on $\partial V$ it holds that $u^{*}=u^{**}$, we deduce that
\begin{align*}
v(t)\ge u^{**}(tx+(1-t)y)
\end{align*}
on the part that is contained in $V$. Thus by combining the facts, it holds that $v(t)\ge u(tx+(1-t)y)$ i.e.\ the claim follows.
\end{proof}

\begin{lemma}
\label{unifcont}
Let $\{u_\alpha\}_{\alpha\in A}$ be a family of functions defined in $V$.
Assume that the family is equicontinuous in a compact set $K$.
Then, given $\eps>0$, there exists $\delta>0$ such that for every $x,y\in V$ with
$|x-y|<\delta$ and $\dist(x,K)< \delta$ and $\dist(y,K)<\delta$ we have $|u_\alpha(x)-u_\alpha(y)|<\eps$ for every $\alpha\in A$.
\end{lemma}

\begin{proof}
Since $u_\alpha$ is equicontinuous, for every $x\in K$, there exists $\delta_x>0$ such that $|u_\alpha(x)-u_\alpha(y)|<\eps/2$ for every $\alpha\in A$ and $y\in V$ with $|x-y|<\delta_x$.

We consider the following cover of $K$,
\[
K\subset \bigcup_{x\in K} B_{\tilde\delta_x}(x)
\]
where $\tilde \delta_x=\delta_x/3$.
Since $K$ is compact there exists a finite number of points $\{x_i\}_{i=1}^N$ such that 
\begin{equation}
\label{cover}
K\subset \bigcup_{i=1}^N B_{\tilde \delta_{x_i}}(x_i).
\end{equation}
We take $\delta=\min_{i=1,\dots,N}\tilde \delta_i$.

Let $x,y\in V$ with $\dist(x,K),\dist(y,K)<\delta$ and $|x-y|<\delta$, we want to prove that $|u_\alpha(x)-u_\alpha(y)|<\eps$.
Since $\dist(x,K)<\delta$ there exists $\tilde x\in K$ such that $|x-\tilde x|<\delta$.
Then, since we have \eqref{cover}, there exists $i\in\{1,\dots,N\}$ such that $|\tilde x- x_i|<\tilde \delta_{x_i}$.
We get 
\[
|x-x_i|\leq |x-\tilde x|+|\tilde x-x_i|\leq \delta+\tilde \delta_{x_i}\leq 2\tilde \delta_{x_i}\leq \delta_{x_i},
\]
therefore $|u_\alpha(x)-u_\alpha(x_i)|<\eps/2$ for every $\alpha\in A$.

Since $|x-y|<\delta$, we have
\[
|y-x_i|\leq |y-x|+|x-\tilde x|+|\tilde x-x_i|\leq \delta+\delta+\tilde \delta_{x_i}\leq 3\tilde \delta_{x_i}\leq \delta_{x_i},
\]
so $|u_\alpha(y)-u_\alpha(x_i)|<\eps/2$ and we conclude that 
\[
|u_\alpha(x)-u_\alpha(y)|\leq |u_\alpha(x)-u_\alpha(x_i)|+|u_\alpha(x_i)-u_\alpha(y)|<\eps.
\]
for every $\alpha\in A$
\end{proof}

\begin{proof}[Proof of Theorem \ref{thm:cont}]
{
The first step is to obtain the equicontinuity on the boundary of the domain.
We begin by observing that since $\Omega$ is strictly convex, given $y \in \partial\Omega$ there exists a unitary vector $w\in\R^N$, $|w|=1$, such that for every $\delta >0$ there exists $\theta >0$ such that
\[
\{x\in\Omega : \langle w,x-y\rangle <\theta\} \subset B_\delta(y)
\]
see Lemma 6 in \cite{BCMR}.
Then given a continuous boundary datum $g$, a boundary point $y \in \partial\Omega$, and $\eps>0$, we can consider $\delta>0$ arising from the continuity of $g$ at $y$ so that 
\[
\eps+\frac{\langle w,x-y\rangle}{\theta}\max g
\quad \text{and} \quad
-\eps+\frac{\langle w,x-y\rangle}{\theta}\min g
\]
are barriers.
Since these functions are affine they are both convex and quasiconvex, we get that both $u_0^*$ and $u_1^*$ are continuous at the boundary.
Since $\partial \Omega$ is compact both $u_0^*$ and $u_1^*$ are equicontinuous on $\partial \Omega$.
Finally, since $u_1^*\leq u_\alpha^*\leq u_0^*$ we conclude that the family $\{u_\alpha^*\}_{\alpha\in[0,1]}$ is equicontinuous on the boundary.
}

Now, let $\eps>0$, since $\partial\Omega$ is compact and $\{u_\alpha^*\}_{\alpha\in[0,1]}$ is equicontinuous at the points of that set, by Lemma~\ref{unifcont}, we have that there exists $\delta>0$ so that
\begin{equation}
\label{estimate}
|u^*_\alpha (x) - u^*_\alpha(y)| < \eps
\end{equation}
whenever $\abs{x-y}<\delta$ and $\dist(x,\partial\Omega),\dist(y,\partial\Omega)<\delta$ for every $\alpha\in[0,1]$.

Now we show that in fact 
\begin{equation}
|u^*_\alpha (x) - u^*_\alpha(y)| < \eps
\end{equation}
whenever $\abs{x-y}<\delta/2$ for every $\alpha\in[0,1]$.
We consider a slightly smaller domain
\[
\tilde\Om =\{z\in \Om\,:\,\dist(z,\partial \Om)> \delta/2\}
\]
and
\[
\Gamma=\overline{\Omega}\setminus\tilde\Omega =\{z\in \overline\Om\,:\,\dist(z,\partial \Om)\leq \delta/2\}.
\]
Suppose that $x,y \in \Om$ with $\abs{x-y}<\delta/2$.
If $x$ or $y$ are in $\Gamma$ then we have $\dist(x,\partial\Omega) < \delta$ and $\dist(y,\partial\Omega)<\delta$ and hence we can employ \eqref{estimate}.
Finally, let $x,y\in\tilde\Om $ such that $\abs{x-y}<\delta/2$, we want to prove that $u^*_\alpha(y)\leq u^*_\alpha(x)+\eps$.
We consider 
\[
u(z)=u^*_\alpha(z-x+y)-\eps,
\]
it is defined in $\tilde\Omega$ and it is $\alpha$-convex.
Since $u^*_\alpha$ is the $\alpha$-convex envelope in $\tilde\Omega$ of it self restricted to $\partial\tilde\Omega$ (see Lemma~\ref{starstar}, observe that $\tilde\Omega$ is strictly convex), and $u\leq u^*_\alpha$ on $\partial\tilde\Omega$, we get $u\leq u^*_\alpha$ in $\tilde\Omega$.
In particular, evaluating at $x$, we get $u^*_\alpha(y)-\eps=u(z)\leq u^*_\alpha(x)$ as we wanted to prove.
\end{proof}

	\subsection{Comparison principle for $\mathcal{L}_\alpha$}
	
	Here our goal is to prove that the equation 
	$$
	\displaystyle 
	 \mathcal{L}_\alpha  u (z) =  \inf_{|v|=1} 
		\{\alpha \langle D^2 u(z) v, v \rangle + (1-\alpha)  |\langle D u(z), v \rangle |^2\} =0.
		$$
		has a comparison principle between classical sub and supersolutions
		(testing with $N$-dimensional functions). This subsection does not assume convexity of the domain.
		
		Let us now state a second definition of viscosity sub/supersolution to the equation
(notice that here we test with $N$-dimensional test functions).
	\begin{definition}[Viscosity solutions using $N$-dimensional test functions] \label{defi-visc.N} \
		An upper semicontinuous function $u\colon\Om\to\mathbb{R}$  is a 
		\textit{viscosity subsolution}  of \eqref{eq:Lu}
				if  for any $z\in \Omega,$ 
				 any test function $\phi\in C^{2}(\Om)$
		such that 
		$$
		\phi(z)= u(z)\text{ and }\phi(y)\ge u(y) \text{ in } \Omega,
		$$
		we have
		\[
		\mathcal{L}_{\alpha} \phi (z)\geq 0.
		\]

		A lower semicontinuous function $u\colon\Om \to\mathbb{R}$  is a 
		\textit{viscosity supersolution}  of 
		\eqref{eq:Lu} 
		if for any 
		$z\in \Omega,$ 
		any test function $\phi\in C^{2}(\Om)$ 
		such that $$\phi(z)= u(z) \text{ and }\phi(y)\le u(y) \text{ in } \Omega,$$ 
we have		
		\[
			\mathcal{L}_{\alpha} \phi (z)\leq 0 .
		\]

		Finally, we say that $u$ is a viscosity solution of \eqref{eq:Lu} when it is both a 
		viscosity subsolution and a viscosity supersolution of	\eqref{eq:Lu}.
	\end{definition}

		First, we show that we can perturb a subsolution in order to make it a 
		strict subsolution. This is one of the key steps in the proof of the comparison principle. 
		
		\begin{lemma} \label{lema.perturbacion.estrcto}
		Let $\alpha\neq 0$ and $\underline{u}$ be a subsolution to $ \mathcal{L}_\alpha  u (z)= 0$, 
		that is, it holds that
	$$
	\displaystyle 
	 \mathcal{L}_\alpha \underline{u} (z) =  \inf_{ |v|=1} \Big\{
		\alpha \langle D^2 \underline{u}(z) v, v \rangle + (1-\alpha)  |\langle D \underline{u}(z), v \rangle |^2 \Big\} \geq 0,		
		$$
		in the viscosity sense (testing with $N$-dimensional test functions). 
		Then, for $k>1$, $\delta>0$ with $\delta << k-1$ the function
		$$
		u_{k,\delta} (z) = k \underline{u} (z) + \delta |z|^2
		$$
		is a strict subsolution, that is, there exists $c>0$ such that
		$$
	\displaystyle 
	 \mathcal{L}_\alpha  u_{k,\delta}  (z) =  \inf_{ |v|=1} \Big\{
		\alpha \langle D^2 u_{k,\delta}  (z) v, v \rangle + (1-\alpha) |\langle D u_{k,\delta}  (z), v \rangle |^2 \Big\}\geq c > 0,	
		$$
		in the viscosity sense (testing with $N$-dimensional test functions). 
		\end{lemma}
		
		\begin{proof}
		We have that 
		$$
		D u_{k,\delta}  (z) = k D \underline{u} (z) + 2\delta z
		$$
		and
		$$
		D^2 u_{k,\delta}  (z) = k D^2 \underline{u} (z) + 2\delta I.
		$$
		Therefore, we have (the following computations can be justified in the viscosity sense)
		$$
		\begin{array}{l}
\displaystyle
		 \mathcal{L}_\alpha  u_{k,\delta}  (z) =  \inf_{|v|=1} \{
		\alpha \langle D^2 u_{k,\delta}  (z) v, v \rangle + (1- \alpha)  |\langle D u_{k,\delta}  (z), v \rangle |^2\} \\[7pt]
		\displaystyle
		= \inf_{ |v|=1} \{
		\alpha \langle  ( k D^2 \underline{u} (z) + 2\delta I ) v, v \rangle 
		+ (1-\alpha)  |\langle k D \underline{u} (z) + 2\delta z, v \rangle |^2\} \\[7pt]
		\displaystyle =
		\inf_{ |v|=1} 
		\{\alpha k \langle D^2 \underline{u} (z)  v, v \rangle 
		+ (1-\alpha)  k^2 |\langle D \underline{u} (z),v \rangle |^2  \\[7pt]
		\displaystyle \qquad 
		+ (1-\alpha) 4 k \delta \langle D \underline{u} (z),v \rangle \langle  z, v \rangle
		+ (1-\alpha) 4 \delta^2 |\langle  z, v \rangle |^2 
		+ 2\delta \alpha\} \\[7pt]
		\displaystyle 
		\geq 
				\inf_{ |v|=1} 
		\{(1-\alpha)  (k^2-k) |\langle D \underline{u} (z),v \rangle |^2  \\[7pt]
		\displaystyle \qquad 
		+ (1-\alpha) 4 k \delta \langle D \underline{u} (z),v \rangle \langle  z, v \rangle
		+ (1-\alpha) 4 \delta^2 |\langle  z, v \rangle |^2 
		+ 2\delta \alpha\} \\[7pt]
		\displaystyle 
		=
		\inf_{ |v|=1} \{
		(1-\alpha)   \Big| \sqrt{k^2-k} \langle D \underline{u} (z),v \rangle
		+ 2 \delta \frac{k}{ \sqrt{k^2-k}} \langle  z, v \rangle \Big|^2  
		\\[7pt] \displaystyle \qquad + (1-\alpha) 4 \delta^2 \Big(1- \frac{k}{ k-1 } \Big) |\langle  z, v \rangle |^2 
		+ 2\delta \alpha\} \\[7pt]
		\geq \delta \alpha >0
		\end{array}
		$$
		provided that 
		\[
		(1-\alpha) 4 \delta^2 \Big(1- \frac{k}{ k-1 } \Big) |\langle  z, v \rangle |^2 
		+ \delta \alpha\geq 0.
		\]
		This is immediate for $\alpha=1$, while for $\alpha\neq 1$ we take $\delta$ such that
				$$
		\frac{1}{\displaystyle 4  \max_{z \in \overline{\Omega}} 
		|\langle  z, v \rangle |^2 } (k-1) \Big(\frac{\alpha}{1-\alpha} \Big) \geq \delta.
		$$
		This ends the proof.
		\end{proof}
		
		With this result at hand we can prove a comparison lemma. 
		
		\begin{lemma} \label{teo.compar}
		Let $\alpha\neq 0$, $\overline{u}:\overline{\Omega} \to \mathbb{R}$ and $\underline{u}
		:\overline{\Omega} \to \mathbb{R}$ be a supersolution and a  subsolution to 
		$$
	\displaystyle 
	 \mathcal{L}_\alpha  u (z) =  \inf_{ |v|=1} \Big\{
		\alpha \langle D^2 u(z) v, v \rangle + (1-\alpha)  |\langle D u(z), v \rangle |^2\Big\} =0.		
		$$
		Then, if $\overline{u} \geq \underline{u}$ on $\partial \Omega$ we get that
		$$
		\overline{u} \geq \underline{u}, \qquad \mbox{ in } \Omega.
		$$
		\end{lemma}
		\begin{proof} We argue by contradiction.
		Assume that  
		$$
		\max_{z \in \Omega}\{ \underline{u} (z) - \overline{u} (z)\} \geq \theta >0.
		$$
		From Lemma \ref{lema.perturbacion.estrcto} we can assume that $ \underline{u}$ is a strict subsolution. 
		
		By taking the supremal convolution of $\underline{u}$,
		$$
		\underline{u}_\gamma (z) = \sup_{y \in \Omega } \Big(  \underline {u} (y) - \frac{1}{2\gamma} 
		|z-y|^2 \Big),
		$$
		and the infimal convolution of $\overline{u}$, 
		$$
		\overline{u}_\gamma (z) = \inf_{y \in \Omega } \Big(  \overline {u} (y) + \frac{1}{2\gamma}
		|z-y|^2 \Big),
		$$
		we obtain 
		a pair of a strict subsolution and a supersolution that are semiconvex and semiconceve and such that
		the difference attains a positive maximum inside $\Omega$, 
		$$
		\max_{z \in \Omega}\{ \underline{u}_\gamma (z) - \overline{u}_\gamma (z) \}\geq \theta >0.
		$$
		
		Now we need the following statement (see Lemma 5.2 in \cite{BGJ13}).
Let $W$ be semiconvex and $V$ be semiconcave. Suppose that $z_0 \in \Omega$ is
a maximum point of $W - V$ such that
$$
W(z_0) -V (z_0) - \max_{\partial \Omega}
(W - V ) > 0.
$$
Then there exists a sequence of points $z_k \to z_0$,  
and  $(p_k, A_k) \in J^{2,+}W(z_k)$, and
$(q_k, B_k) \in J^{2,-}V (z_k)$ and a constant $K > $0, which may depend on the semiconvexity
constant of $W - V$, such that
$$
 |p_k - q_k| \to 0 \qquad \mbox{as }, k \to \infty,
 $$
 $$
 A_k - B_k \leq - \frac{K}{k} I .
 $$
Here $J^{2,+}$ ad $J^{2,-}$ denote the semijets.

Using this statement with $W= \underline{u}_\gamma $ and $V=  \overline{u}_\gamma$ 
we obtain a sequence $z_k \to z_0$, $(p_k,A_k)$ and $(q_k, B_k)$, such that
the previous conditions hold and 
$$
\min_{|v|=1} \{\alpha \langle A_k v,v \rangle + (1-\alpha)| \langle p_k ,v \rangle |^2 \}\geq c >0, 
$$
and 
$$
\min_{|v|=1}\{ \alpha \langle B_k v,v \rangle + (1-\alpha)| \langle q_k ,v \rangle |^2 \}\leq 0.
$$
Take $v_k$ such that the minimum is achieved in the second inequality, so that
$$
 \alpha \langle B_k v_k,v_k \rangle + (1-\alpha)| \langle q_k ,v_k \rangle |^2 \leq 0
$$
and it always also holds that
$$
 \alpha \langle A_k v_k,v_k \rangle + (1-\alpha)| \langle p_k ,v_k \rangle |^2 \geq  c > 0.
$$
Now, using
$$
 A_k - B_k \leq - \frac{K}{k} I ,
 $$
  we get
  $$
  \begin{array}{l}
  \displaystyle 
  \alpha \langle  \Big( A_k + \frac{K}{k} I  \Big) v_k,v_k \rangle
  + (1-\alpha)| \langle p_k ,v_k \rangle |^2 \\[7pt]
  \displaystyle 
  = \alpha \langle A_k v_k,v_k \rangle + \frac{K}{k}  + (1-\alpha)| \langle p_k ,v_k \rangle |^2
  \\[7pt]
  \displaystyle 
  \leq  \alpha \langle B_k v_k,v_k \rangle + (1-\alpha)| \langle q_k ,v_k \rangle |^2
  + (1-\alpha) \Big( | \langle p_k ,v_k \rangle |^2 -  |\langle q_k ,v_k \rangle |^2  \Big).
  \end{array}
  $$

Hence, we obtain
$$
\begin{array}{l}
\displaystyle 
c + \frac{K}{k} \leq \alpha \langle A_k v_k,v_k \rangle + \frac{K}{k}  + (1-\alpha)| \langle p_k ,v_k \rangle |^2
\\[7pt]
\displaystyle
\qquad \leq (1-\alpha) \Big( | \langle p_k ,v_k \rangle |^2 -  |\langle q_k ,v_k \rangle |^2  \Big).
\end{array}
$$
In general the constant of semiconvexity $K = K(\gamma) \to 0$ as $\gamma \to 0$. Therefore, sending 
$\gamma\to 0$ and $k\to \infty$ we reach a contradiction. 	
		\end{proof}
		
		Using the comparison lemma we get uniqueness of solutions.

\begin{corollary} \label{teo.uni} Let $\alpha\neq 0$.
		There exists a unique viscosity solution (testing with $N$-dimensional test functions) to 
		\begin{equation} \label{ooo}
		\left\{
		\begin{array}{l}
	\displaystyle 
	 \mathcal{L}_\alpha  u (z) =  \inf_{ |v|=1} \Big\{
		\alpha \langle D^2 u(z) v, v \rangle + (1-\alpha)  |\langle D u(z), v \rangle |^2
		\Big\} =0, 
		\mbox{ in } \Omega, \\[7pt]
		u(z)=g (z), \qquad \mbox{ on } \partial \Omega.
		\end{array}
		\right.	
		\end{equation}
		\end{corollary}

 %%%%%%%%%%
\begin{remark}
\label{rem:conincide} At this point we highlight that both notions of solution to the PDE coincide. 
	Remark that, as we already explained before Definition \ref{defi-visc}, if $u$ is a viscosity solution testing with $1$-dimensional from above and $N$-dimensional from below then it is a viscosity solution testing with $N$-dimensional functions from both sides. Since we have uniqueness of viscosity solutions testing with $N$-dimensional functions from both sides we obtain that both notions of viscosity solutions
	are equivalent whenever $\alpha\neq 0$. 
		
	{ In fact, assume that we have a continuous solution $u$ to the PDE testing with $N$-dimensional functions. Take a ball $B_R$ inside $\Omega$ and look for the convex envelope inside the ball of $u|_{\partial B_R}$ (call this convex envelope $v$). Then $v$, as the ball is strictly convex, $v$ is a solution of the PDE inside $B_R$ testing $1$-dimensional functions (this follows from our results since $v$ is the $\alpha-$convex envelope of $u|_{\partial B_R}$ in $B_R$, see below). Then, $v$ is also a solution testing with 
	$N$-dimensional functions and, from the comparison result (that implies uniqueness) we obtain that $u=v$ in $B_R$. Therefore, it turns out that $u$ is also a solution to the PDE in $B_R$ testing with 1-dimensional functions. Hence, both notions of solution coincide, a continuous function is a solution testing with $N$-d functions if and only if it is a solution testing with $1$-d functions.	
	
	Notice that we are using that the domain is strictly convex just to obtain continuity
	up to the boundary of the $\alpha$-convex envelope of a continuous boundary datum, however the fact that the $\alpha$-convex envelope is a solution to the PDE inside $\Omega$ 
	is a local property that holds testing both with $N$-d or with $1$-d functions regardless of the strict convexity of $\Omega$. 
	} 
\end{remark}
%%%%%%%%%%	

%
Now, we show that the $\alpha$-convex envelope is in fact a solution to the PDE
\begin{equation} \label{convex-envelope-X.5567999}
	\mathcal{L}_\alpha  u (z) := 
	 \inf_{
	\substack{x,y \in \overline{\Omega}  \\
		z=t_0 x+ (1-t_0)y} } L_\alpha (u) (t_0) = 0, \qquad z \in \Omega,
		\end{equation}
		in the viscosity sense given in Definition \ref{defi-visc}.

\subsection{Proof of Theorem \ref{teo.2.intro}: viscosity solutions and the $\alpha$-convex envelope}

\begin{proof}[Proof of Theorem \ref{teo.2.intro}]
The statement of continuity in $\overline{\Omega}$ follows from Theorem~\ref{thm:cont}.
Then we recall that the supremum of $\alpha$-convex functions is $\alpha$-convex
as shown in Lemma \ref{lema.sup}. 
By Theorem \ref{teo.1.intro} we obtain
	that $u_\alpha^*$ verifies 
	\begin{equation} \label{convex-envelope-X.5567}
	\mathcal{L}_\alpha  u (z) := 
	 \inf_{
	\substack{x,y \in \overline{\Omega}  \\
		z=t_0 x+ (1-t_0)y} } L_\alpha (u) (t_0) \geq 0, \qquad z \in \Omega,
		\end{equation}
	touching with $1$-dimensional test functions from above.

	To show the reverse inequality we argue by contradiction. Assume that there exists 
	a point $z_0 \in \Omega$ and a smooth $N-$dimensional test function $\psi$ such that
	$u_\alpha^* - \psi$ has a strict minimum at $z_0$ with 
	\begin{equation} \label{convex-envelope-X.55678} 
	 \inf_{
	\substack{x,y \in \overline{\Omega}  \\
		z_0=t_0 x+ (1-t_0)y} } L_\alpha (\psi) (t_0) \geq c > 0.
		\end{equation}
		As $\psi$ is smooth there exists $r>0$ small such that 
		\begin{equation} \label{convex-envelope-X.556789}
	 \inf_{
	\substack{x,y \in \overline{\Omega}  \\
		z=t_0 x+ (1-t_0)y} } L_\alpha (\psi) (t_0) \geq \frac{c}{2} > 0, \qquad \text{ for all } z \in B_r (x_0) \subset \Omega.
		\end{equation}
		Now we take 
		$$
		\hat{u} (x) = \max \big\{ u_\alpha^*(x) ; \psi (x) + \delta \big\}
		$$
		with $\delta >0 $ small such that $u_\alpha^*(x) > \psi (x) + \delta$ for every $x \in \Omega \setminus B_r (z_0)$.
		This function $\hat{u}$ is $\alpha$-convex since it coincides with $u_\alpha^*$ in $\Omega \setminus B_r (z_0)$ and both $u_\alpha^*$ and $\psi$ are $\alpha$-convex in $B_r (z_0)$. Then, we have that $\hat{u}$ is $\alpha$-convex with $\hat{u}|_{\partial \Omega}
		\leq g$  and
		$$
		\hat{u} (z_0) = \max \big\{ u_\alpha^*(z_0) ; \psi (z_0) + \delta \big\} =\psi (z_0) + \delta > u_\alpha^* (z_0).
		$$
This contradicts the fact that $u_\alpha^*$ is defined as the supremum in \eqref{convex-envelope-X.99}. 

Finally the theorem follows from the uniqueness of solution for the equation with the Dirichlet boundary condition.
\end{proof}

\subsection{Proof of Theorem \ref{th.continu}: A bridge between convex and quasiconvex envelopes}
	\label{sect-bridge}

	In this section we show that the $\alpha$-convex envelopes $u^*_\alpha$  in Definition \ref{def:convex-env} 	constitute a non-increasing and continuous (with respect to $\sup$-norm) curve $\alpha \mapsto u^*_\alpha$. Moreover, it goes from quasiconvexity ($u^*_0$ is the quasiconvex envelope) to convexity ($u^*_1$ is the convex 
	envelope). 
	
%%%%%%%%%%%%%%%%%%%%%%%%%%%%%%%%
	\begin{proof}[Proof of Theorem \ref{th.continu}]
	First, let us show that the $\alpha$-convex envelopes are non-increasing with $\alpha$. 
	
	Take $u_\alpha$ an $\alpha$-convex function. Then by Proposition \ref{prop:bigger-alpha-stronger}, $u_{\alpha}$ is also  $\hat \alpha$-convex function for $\hat \alpha<\alpha$. 
	From this fact it follows that
	$$
	\begin{array}{l}
	\displaystyle 
	u^*_\alpha (x) = \sup \Big\{v(x) : v \mbox{ is $\alpha-$convex and verifies } v|_{\partial \Omega} \leq g \Big\}
	 \\[7pt]
	 \displaystyle \quad \qquad  \leq \sup \Big\{v(x) : v \mbox{ is $\hat \alpha-$convex and verifies } v|_{\partial \Omega} \leq g \Big\} 
	=u^*_{\hat{\alpha}} (x)
	\end{array}
	$$
	and then we conclude that $\alpha \mapsto u^*_\alpha$ is non-increasing. By using Theorem \ref{thm:cont} together with Arzel\`a-Ascoli's theorem, and passing to a subsequence if necessary, we get that 
	 $$
	\hat{w} (z) = \lim_{j\to \infty }  u^*_{\alpha_j} (z)
	$$
	uniformly to a continuous limit on $\overline \Om$ where $\alpha_j\to\hat\alpha$.	Moreover, 	as $u^*_{\alpha_j} $ is $\alpha_j$-convex we have that
	\begin{equation} \label{late}
	u^*_{\alpha_j} (tx+(1-t)y) \leq v_j (t)
	\end{equation}
	with $v_j$ the solution to 
	\begin{equation} \label{coffe.33}
	\left\{
	\begin{array}{l}
\displaystyle L_{\alpha_j} (v_j) (t) =\alpha_j v_{j} '' (t) + (1-\alpha_j) |v_j'(t)|^2 =0, \qquad t \in (0,1), \\
\displaystyle v_j (0) = u_{\alpha_j} (y), \\ 
v_j (1) = u_{\alpha_j} (x),
\end{array}
\right.
\end{equation}
	and since the equation is stable by  Remark~\ref{rem:continuity1-d} we get
	$$
\hat{w} (tx+ (1-t)y) \leq \hat{v} (t)
$$
{where $\hat{v}$ satisfies
\begin{equation} \label{coffe.335}
	\left\{
	\begin{array}{l}
\displaystyle L_{\hat \alpha} (\hat v) (t)=0, \qquad t \in (0,1), \\
\displaystyle \hat v (0) = \hat{w} (y), \\ 
\hat v (1) = \hat{w}(x).
\end{array}
\right.
\end{equation}}
Thus the limit is $\hat{w}$ is $\hat{\alpha}$-convex.

	{\bf First case: $\alpha_j$ increasing.} 		Now, we just observe that $u^*_{\hat{\alpha}} \leq u^*_{\alpha_j}$
		and passing to the limit as $\alpha_j \to \hat{\alpha}$ we obtain 
		$$
		u^*_{\hat{\alpha}} (z) \leq \hat{w}(z)  \qquad z \in \overline{\Omega}.
		$$
		Since $u^*_{\hat{\alpha}}$ is defined as a supremum of $\hat{\alpha}$-convex functions, we conclude that $$u^*_{\hat{\alpha}} (z) \equiv \hat{w}(z)$$ and hence
		$$
		\lim_{\alpha_j \to \hat{\alpha}} u^*_{\alpha_j} (z) = u^*_{\hat{\alpha}} (z).
		$$
		
		{\bf Second case: $\alpha_j$ decreasing with limit $ \hat{\alpha}\neq 0$.} 
	Since now $u^*_{\alpha_j}\le u^*_{\hat \alpha}$, we immediately have
	\begin{align}
\label{eq:dec-a-from-above}
	 \lim_{j\to \infty }  u^*_{\alpha_j} (z) = \hat{w} (z)  \leq u^*_{\hat{\alpha}} (z), \qquad z \in \overline{\Omega}, 
\end{align}
and it remains to show that $\hat{w} \geq u^*_{\hat{\alpha}}$.
Let $\psi$ be a smooth $N-$dimensional function that touches $\hat{w}$
from below at $z_0 \in \Omega$. 
From the 
uniform convergence of $ u_{\alpha_n}^*$ to $\hat{w}$ we obtain the existence of a sequence of points with
$z_n \to z_0$ and such that $\hat{w} - \psi$ has a minimum at $z_n$. 
Next we denote
 $$
 (L_{\alpha_n})_v (u_{\alpha_n}) (z):=\alpha_n \langle D^2 u_{\alpha_n}(z) v, v \rangle \! + \! (1-\alpha_n)  |\langle D u_{\alpha_n}(z) ,v \rangle |^2. 
 $$
Now, using that $u^*_{\alpha_n}$ is  by Theorem \ref{teo.2.intro} a 
viscosity solution to
\begin{equation} \label{convex-envelope-X.5562288}
	\inf_{|v|=1} (L_{\alpha_n})_v (u_{\alpha_n}) (z) =0, \qquad z \in \Omega,
		\end{equation}
we obtain that  
\begin{equation} \label{convex-envelope-X.556228899}
	\inf_{|v|=1} (L_{\alpha_n})_v (\psi) (z_n) \leq 0.
		\end{equation}
		Since $\psi$ is smooth, the infimum is attained.
	This means that there is a sequence of directions $v_n$, $|v_n | =1$, such that
	$$
	(L_{\alpha_n})_{v_n} (\psi) (z_n) \leq 0.
	$$
	We can assume (extracting a subsequence if necessary) that $v_n \to v_\infty$ with $|v_\infty|=1$.
	Passing to the limit we obtain that 
	$$
	(L_{\hat \alpha})_{v_\infty} (\psi) (z_0) \leq 0,
	$$
	and hence we get 
	\begin{equation} \label{convex-envelope-X.556228899.77}
	\inf_{|v|=1} (L_{\hat \alpha})_v (\psi) (z_0) \leq 0.
		\end{equation}
		Then, we use the comparison principle Lemma \ref{teo.compar} to get $\hat{w} \ge u^*_{\hat{\alpha}}$. Altogether, we have shown
		$$
		\hat{w} \equiv u^*_{\hat{\alpha}}.
		$$

	{\bf Third case: $\alpha_j$ decreasing with $\alpha_j \to 0$.}
As before, we consider
	$$
	\hat{w} (z) = \lim_{j\to \infty }  u^*_{\alpha_j} (z). 
	$$

Since $u^*_{\alpha_j}\le u^*_{0}$, it follows that $\hat w\le u^*_{0}$. 
To establish the reverse inequality, we use the comparison principle for quasiconvex viscosity sub-/supersolution, Corollary 5.6 in \cite{BGJ13} instead of Lemma~\ref{teo.compar}.
Let $\psi$ be a smooth $N-$dimensional function that touches $\hat{w}$
from below at $z_0 \in \Omega$. 
From the 
uniform convergence of $ u_{\alpha_n}^*$ to $\hat{w}$ we obtain the existence of a sequence of points with
$z_n \to z_0$ and such that $\hat{w} - \psi$ has a minimum at $z_n$. 
Next we denote $(L_{\alpha_n})_v (u_{\alpha_n}) (z):=\alpha_n \langle D^2 u_{\alpha_n}(z) v, v \rangle \! + \! (1-\alpha_n)  |\langle D u_{\alpha_n}(z) ,v \rangle |^2 $.
Now, using that $u^*_{\alpha_n}$ is  by Theorem \ref{teo.2.intro} a 
viscosity solution to
\begin{equation} 
	\inf_{|v|=1} (L_{\alpha_n})_v (u_{\alpha_n}) (z) =0, \qquad z \in \Omega,
		\end{equation}
we obtain 
\begin{equation} 
	\inf_{|v|=1} \alpha_n \langle D^2\psi(z_n) v, v \rangle+\! (1-\alpha_n)  |\langle D \psi(z_n) ,v \rangle |^2 \leq 0.
		\end{equation}
			Since $\psi$ is smooth, the infimum is attained.
	This means that there is a sequence of directions $v_n$, $|v_n | =1$, such that
	$$
	(L_{\alpha_n})_{v_n} (\psi) (z_n) \leq 0.
	$$
	We can assume (extracting a subsequence if necessary) that $v_n \to v_\infty$ with $|v_\infty|=1$.
	Passing to the limit we obtain that 
	$$
	|\langle D \psi(z_0) ,v_{\infty} \rangle|^2=(L_{0})_{v_\infty} (\psi) (z_0) \leq 0.
	$$
	 Since
	 \begin{align*}
 0\ge \alpha_n \langle D^2\psi(z_n) v_n, v_n \rangle+\! (1-\alpha_n)  |\langle D \psi(z_n) ,v_n \rangle |^2 \ge \langle D^2\psi(z_n) v_n, v_n \rangle
\end{align*}
and $\psi$ is smooth, 
it also follows that 
\begin{align*}
\langle D^2\psi(z_0) v_{\infty}, v_{\infty} \rangle\le 0.
\end{align*}	
Thus we have shown
\begin{align*}
 \min_{ \substack{v \colon |v|=1, \\
		\langle v,  D  \psi(z_0) \rangle =0}} 
		\langle D^2 \psi(z_0) v, v \rangle \le 0,
\end{align*}
i.e.\  $\hat w$ is a quasiconvex viscosity supersolution with boundary values $g$. 
		Since $u^*_0$ is the quasiconvex solution to the same equation with the boundary values $g$, the comparison principle, Corollary 5.6 in \cite{BGJ13}, implies $\hat{w} \ge u^*_{0}$. Altogether, we have shown
		$$
		\hat{w} \equiv u^*_{0}. \qedhere
		$$
		
			\end{proof}
%%%%%%%%%%%%%%%%%%%%%%%%

\section{Regularity of the $\alpha$-convex envelope} \label{sect-regul}

Now our goal is to prove that the $\alpha$-convex envelope is $C^1$ for $\alpha\neq 0$ in a strictly convex domain a suitable boundary data.  
This result could be compared to the $C^{1,\alpha}$-result of Oberman and Silvestre \cite{OS} for the convex envelope. We follow a similar path but our result is different and we make a weaker assumption on the boundary data.

First, we introduce the following definition to state our assumptions on the boundary data. The idea of this assumption is to prevent a wedge-like behavior propagating from the boundary into the domain: we call the desired behaviour, i.e.\ that the boundary function cannot be touched from below by a V-shaped function, by NV (Not V).
%%%%%%%%%%%%%%%%%%%%%%%%%%%%%
\begin{definition}[V-shaped touching]
\label{eq:touch-v}
Given $g:\partial\Omega\to\R$ we say that we can touch $g$ from below at $x_0$ with a V shaped function if there exist two smooth functions $\pi$ and $\tilde \pi$ such that $g(x_0)=\pi(x_0)=\tilde\pi(x_0)$, $ D \pi(x_0)\neq D \tilde\pi(x_0)$, and for some $r>0$ we have 
\[
\max\{\pi(x),\tilde\pi(x)\}\leq g(x)
\]
for every $x\in B_r(x_0)\cap \partial\Omega$ and $B_r(x_0)\cap\{x\in \Omega:\pi(x)=\tilde\pi(x)\}=L\cap B_r(x_0)\cap\Omega\neq \emptyset$ for some hyperplane $L$.
\end{definition}
%%%%%%%%%%%%%%%%%%%%%%%%%%%%%%%%%%
\begin{definition}[NV-property] \label{NV-property}
We say that $g:\partial\Omega\to\R$ is NV or has the NV-property if it cannot be touched from below at any point with a V shaped function according to Definition~\ref{eq:touch-v}. 
\end{definition}
%%%%%%%%%%%%%%%%%%%%%%%%%%%%%%%%%%

We illustrate the need of the last condition in the definition of V-shaped touching in the following example.

\begin{example} {\rm
We consider $\Omega=B_1\subset \R^2 $.
The function $g(x,y)=x-1$ is NV (this will follow from Lemma~\ref{difNV}).
Consider $\pi(x)=x-1$ and $\tilde \pi(x)=2-x$.
This pair does not constitute a V-shaped touching since $\{x\in \Omega:\pi(x)=\tilde\pi(x)\}=\emptyset$.}
\end{example}

A differentiable function defined in an open set cannot be touched from below at any point with a V shaped function. Since $g$ is only defined in $\partial\Omega$, the definition of NV does not only involve the regularity of $g$ but also the regularity of the boundary of the domain. We illustrate the undesirable behavior in the following examples.

\begin{example} {\rm 
We consider $\Omega=B_1\subset \R^2$.
The function $g(x,y)=|x|$ it is not NV since we can touch it with a V shape function.
We can verify this, according to the Definition~\ref{eq:touch-v}, by considering $\pi(x)=x$ and $\tilde \pi(x)=-x$.}
\end{example}

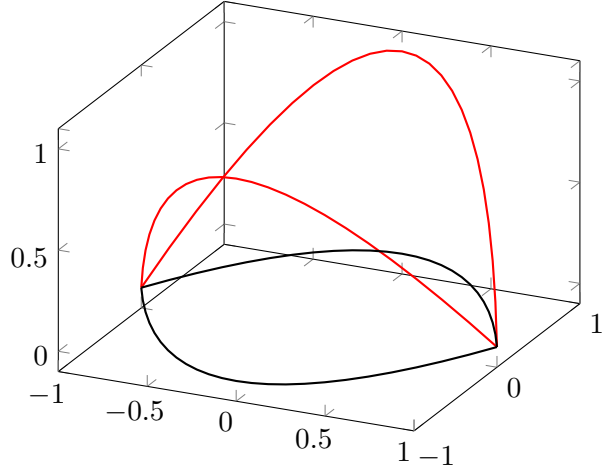
\begin{figure}
\begin{tikzpicture}
\begin{axis}
\addplot3[thick,red,variable=x,domain=-1:1,samples y=0] (x,1-x^2,1-x^2);
\addplot3[thick,red,variable=x,domain=-1:1,samples y=0] (x,x^2-1,1-x^2);
\addplot3[thick,black,variable=x,domain=-1:1,samples y=0] (x,1-x^2,0);
\addplot3[thick,black,variable=x,domain=-1:1,samples y=0] (x,x^2-1,0);
\end{axis}
\end{tikzpicture}
\caption{The function in Example~\ref{ex}.}
\label{fig:ex}
\end{figure}

\begin{example}
\label{ex}
{\rm
We consider $\Omega=\{(x,y)\in\R^2:x^2-1\leq y\leq 1-x^2\}$ and $g(x,y)=1-x^2$, see Figure~\ref{fig:ex}.
Even though $g$ is smooth it is not NV.
This follows immediately by observing  that $g(x,y)=|y|$ on $\partial\Omega$.}
\end{example}

\begin{definition}
We say that $g$ is \textit{upper differentiable} at $x_0$ if there exists a linear transformation $T$ such that
\[
g(x)-g(x_0)-T(x-x_0)\leq o(|x-x_0|).
\] 
\end{definition}

Observe that for a differentiable function we have that there exists $T$ such that
\[
|g(x)-g(x_0)-T(x-x_0)|\leq o(|x-x_0|).
\] 
So, in particular it is upper differentiable.

Also observe that an upper differentiable function may not be differentiable.
For example $g(x)=-|x|$ is upper differentiable at 0 since the inequality holds for $T\equiv 0$.

\begin{lemma}
\label{comp}
If $f$ and $h$ are differentiable at $x_0$, $g$ is upper differentiable at $h(x_0)$, $f(x_0)= g(h(x_0))$ and 
\[
f\leq g \circ h
\]
in a neighborhood of $x_0$, then $Df(x_0)= T Dh(x_0)$ where
$T$ is the linear transformation given by the definition of upper differentiable for $g$.
\end{lemma}

\begin{proof}
On the one hand, by the definition of diferentiable and upper differentiable, we have
\[
|h(x)-h(x_0)+Dh(x_0)(x-x_0)|\leq o(|x-x_0|)
\]
and
\[
g(y)\leq g(h(x_0))+T(y-h(x_0))+o(|y-h(x_0)|).
\]
So we get, choosing $y=h(x)$,
\[
\begin{split}
g(h(x))&\leq  g(h(x_0))+T((Dh(x_0)(x-x_0)+o(|x-x_0|))\\
&\quad\quad\quad+o(|Dh(x_0)(x-x_0)|+ o(|x-x_0|)|)\\
&=  g(h(x_0))+T Dh(x_0)(x-x_0)+o(|x-x_0|).
\end{split}
\]

On the other hand we have
\[
f(x)\geq f(x_0)+Df(x_0)(x-x_0)+o(|x-x_0|),
\]
and then we get
\[
\begin{split}
f(x)&\leq g \circ h(x)\\
f(x_0)+Df(x_0)(x-x_0)&\leq g(h(x_0))+T Dh(x_0)(x-x_0)+o(|x-x_0|).
\end{split}
\]
Since $f(x_0)= g(h(x_0))$ we obtain
\[
Df(x_0)(x-x_0)\leq T Dh(x_0)(x-x_0)+o(|x-x_0|)
\]
(here we interpret $T$ as the matrix corresponding to the linear transformation so that $T Dh(x_0)$ immediately makes sense)
and we conclude that $$Df(x_0)= T Dh(x_0)$$ as desired. 
\end{proof}

\begin{lemma}
\label{difNV}
Let $\Omega$ be strictly convex with a differentiable boundary and $g$ upper differentiable, then $g$ is NV.
\end{lemma}

\begin{proof}
Suppose not. Then there exists $x_0 \in \partial\Omega$ and two differentiable functions $\pi$ and $\tilde \pi$ such that $g(x_0)=\pi(x_0)=\tilde\pi(x_0)$, $ D \pi(x_0)\neq D \tilde\pi(x_0)$ and for some $r>0$ we have 
\[
\max\{\pi(x),\tilde\pi(x)\}\leq g(x)
\]
for every $x\in B_r(x_0)\cap \partial\Omega$ and $B_r(x_0)\cap\{x\in \Omega:\pi(x)=\tilde\pi(x)\}=L\cap B_r(x_0)\cap\Omega\neq \emptyset$ for some hyperplane $L$.

Since $\partial\Omega$ is differentiable there exists a domain $V\subset\R^{N-1}$ and a regular differentiable function $h:V\to\R^N$ (we consider smaller $r$ if needed) that parametrize $B_r(x_0)\cap \partial\Omega$.
We have 
$$
\pi\circ h,\tilde\pi \circ h\leq g\circ h.
$$
If $T$ is the linear transformation given by the definition of upper differentiability for $g$, by Lemma~\ref{comp} the inequalities imply that at $x_0$ $$ D  \pi Dh=T Dh= D  \tilde \pi Dh.$$ We have that at $x_0$ $$ (D\pi -D\tilde \pi)  Dh= 0.$$
Since $h$ parametrices the boundary and is regular, $Dh$ spans the tangent to the boundary,  so $D\pi -D\tilde \pi$ is normal to the boundary.
The hyperplane $L$  (the set where $\pi$ and $\tilde \pi$ coincide) is normal to $ D \pi-\tilde  D  \pi$, so we get that $L$ has to be tangent to the boundary. 
Since $\Omega$ is strictly convex the tangent plane is outside $\Omega$.
We have reached a contradiction since $L\cap \Omega\neq \emptyset$.
\end{proof}

As an example where this result can be applied we mention the following. 

\begin{example} {\rm
By Lemma~\ref{difNV}, the function $g(x,y)=-|x|$ on $\partial B_1$ is NV even though it is not smooth.}
\end{example}

The NV-property can hold for non smooth domains as illustrated by the following example.

%%%%%%%%%%%%%%%%%%%%%%
\begin{example}
\label{exampleASD}
{\rm
We consider $\Omega=\{(x,y)\in\R^2:x^2-1\leq y\leq 1-x^2\}$ and $g(x,y)=-\sqrt{|y|}$. 
Here we have the same domain as in Example~\ref{ex} but the function $g$ is different.

Let us show that this function is NV. We do that at (1,0), the point (-1,0) is analogous and in the rest of the points the domain is smooth.
Suppose not. Then we have the functions $\pi$, $\tilde \pi$ and the line $L$ given by Definition~\ref{eq:touch-v},but on the other hand it is impossible that such functions $\pi$, $\tilde \pi$ touch $g$ from below  at $(1,0)$ since derivative of $g$ is infinity there.

To be more precise, in a neighborhood of (1,0) we have that $\pi$ or $\tilde \pi$ is larger than the other for the points of the boundary with a positive $y$, suppose $\pi$ is. 
Then $\pi(x,1-x^2)\leq -\sqrt{1-x^2}$ with equality for $x=1$, but this is impossible since the derivative of the RHS at 1 is infinity and the LHS is smooth.

}
\end{example}
%%%%%%%%%%%%%%%%%%%%%%%%%%

This shows that the concept of NV functions is quite rich as it contains
not only smooth functions in smooth domains (there are also other curious examples).

Observe that the singularity of $g$ plays an important role in Example~\ref{exampleASD}. A similar shaped function may fail. 
For example $g(x,y)=-|y|$ is not NV. This can be proved by considering $\pi(x,y)=4(x-1)+y$ and $\tilde \pi(x,y)=4(x-1)-y$ .

The next lemma states existence of a unique supporting $\alpha$-hyperplane for $\alpha$-convex envelope under suitable assumptions. For the definition of  an $\alpha$-hyperplane, see Definition~\ref{def:alpha-hyperplane}.

\begin{lemma} \label{lema-supporting-unique} Let $\alpha\neq 0$.
If $\Omega$ is strictly convex and the boundary data $g$ is continuous and NV then the $\alpha$-convex envelope $u^*_\alpha$ has a unique supporting $\alpha$-hyperplane at every point.
\end{lemma}

\begin{proof}
The idea of the proof is as follows. 
If there are two supporting $\alpha$-hyperplanes, by the NV-property, we may find another $\alpha$-hyperplane such that the supremum of this $\alpha$-hyperplane and the two supporting $\alpha$-hyperplanes would give a strictly larger $\alpha$-convex envelope, a contradiction.

Observe that since $u^*_\alpha$ is $\alpha$-convex we know that there is at least one supporting $\alpha$-hyperplane by Theorem \ref{chess}.
For the sake of contradiction, suppose that there are two different ones at $z_0$.
That is we have 
\[
\pi_\alpha (z) = v(\langle z-z_0, \nu \rangle) 
\quad \text{and}\quad
\tilde\pi_\alpha (z) = \tilde v(\langle z-z_0, \tilde \nu \rangle)
\]
where $\nu,\tilde\nu\in \R^N$, $L_\alpha v= L_\alpha \tilde v=0$, $v(0)=\tilde v(0)=u^*_\alpha(z_0)$ and $\pi_\alpha, \tilde \pi_\alpha\leq u^*_\alpha$.

Using the explicit representation found in Lemma \ref{lemma-L-1d}, we write
\[
v(t)=\frac{1}{K_\alpha}\ln(1+Ct)+u^*_\alpha(z_0)
 \quad \text{and}\quad
\tilde v(t)=\frac{1}{K_\alpha}\ln(1+\tilde Ct)+u^*_\alpha(z_0).
\]
We can select $\nu$ and $\tilde\nu$ (possible by changing them for $-\nu$ and $-\tilde\nu$) such that $C,\tilde C\geq 0$.

Then, we can compute the set where $\pi_\alpha$ and $\tilde \pi_\alpha$ coincide, we have
\[
\begin{split}
\frac{1}{K_\alpha}\ln(1+C\langle z-z_0, \nu \rangle)+u^*_\alpha(z_0)
&=
\frac{1}{K_\alpha}\ln(1+\tilde C\langle z-z_0, \tilde \nu \rangle)+u^*_\alpha(z_0),
\\
C\langle z-z_0, \nu \rangle
&=
\tilde C\langle z-z_0, \tilde \nu \rangle,
\\
\langle z-z_0, C\nu-\tilde C\tilde \nu \rangle
&=0.
\end{split}
\]
Observe that $C\nu\neq\tilde C\tilde \nu$ (otherwise the functions $\pi_\alpha$ and $\tilde \pi_\alpha$ would be the same).

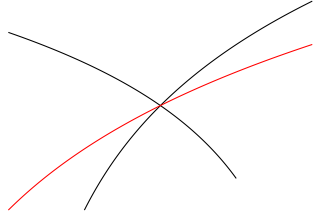
\begin{figure}
     \centering
         \centering
\begin{tikzpicture}[scale=2]
\draw plot [domain=0.5:2,samples=100] (\x,{ln(\x)});
\draw plot [domain=0:1.5,samples=100] (\x,{0.7*ln(2-\x)});
\draw [draw=red] plot [domain=0:2,samples=100] (\x,{ln(\x+1)-ln(2)});
\end{tikzpicture}
	\caption{The function $\pi$ in red.}
	\label{fig:pi}
\end{figure}

Now we construct a third $\alpha$-hyperplane $\pi$, see Figure~\ref{fig:pi}.
We want $\pi=\pi_\alpha=\tilde\pi_\alpha$ in $L$ and 
\[
\pi\leq\max\{\pi_\alpha,\tilde\pi_\alpha\}
\]
and we want the equality to hold only on $L$.

We consider 
\[
\pi(z)=V(\langle z-z_0, \rho \rangle)
\]
where $V(t)=\frac{1}{K_\alpha}\ln(1+kt)+u^*_\alpha(z_0)$, $k>0$ and $\rho$ is a unitary vector.

Observe that, since $\ln (\cdot)$ is a monotone function, the set where $\pi_\alpha$ is larger than $\tilde\pi_\alpha$ is given by 
\[
C\langle z-z_0, \nu \rangle\geq 
\tilde C\langle z-z_0, \tilde\nu\rangle\]
that is, 
\[
\langle z-z_0, C\nu-\tilde C\tilde \nu \rangle
\geq 
0.
\]
There we want to have that $\pi_\alpha$ is larger than $\pi$, and this set is given by $\langle z-z_0, C\nu-k \rho \rangle\geq 0$.
So we need to impose that $C\nu-k \rho=p(C\nu-\tilde C\tilde \nu)$ for some $p>0$.
In a similar way we get $\tilde C\tilde \nu-k \rho=q(\tilde C\tilde \nu-C\nu)$ for some $q>0$.
Thus we take $k \rho=\frac{\tilde C\tilde \nu+C\nu}{2}$ and the equalities hold for $p=q=1/2$.
Observe that since $C\nu\neq\tilde C\tilde \nu$ we have $k \rho \neq \tilde C\tilde \nu$ and $k \rho\neq \tilde C\tilde \nu$.
Therefore $\pi$ only coincide with the other two $\alpha$-hyperplanes in $L$ and we get that $\pi<\max\{\pi_\alpha,\tilde\pi_\alpha\}$ outside $L$. 

Let $x_0\in L\cap \partial\Omega$ such that both $\pi_\alpha$ and $\tilde \pi_\alpha$ are defined there.
If $g(x_0)=\pi_\alpha(x_0)=\tilde\pi_\alpha(x_0)$ this contradicts the NV-property: This contradiction can be obtained by observing that since the segment between $z_0$ and $x_0$ is contained in $\Omega$, the last part of the definition of V-shape touching is verified, $L\cap\Omega$ contains points arbitrarily close to $x_0$.
Thus, we have $g(x_0)>\pi_\alpha(x_0)=\tilde\pi_\alpha(x_0)$ and we can take $\eps>0$ small enough such that $\pi+\eps\leq g$ in a neighborhood of $x_0$.
Since $\partial\Omega$ is compact we can take $\eps>0$ small enough such that $\pi+\eps\leq g$ everywhere.

Finally, we consider $v=\max\{\min g,\eps+\pi,\pi_\alpha,\tilde\pi_\alpha\}$.
By the above construction we have $v\leq g$ everywhere and we reach a contradiction since $v(z_0)=\eps+u_\alpha^*(z_0)>u_\alpha^*(z_0)$.
\end{proof}

Next we prove Theorem~\ref{teo.C1.intro} stating that 
if the boundary datum $g$ is continuous and NV then the $\alpha$-convex envelope, $u^*_\alpha$, is $C^1(\Om)$.

\begin{proof}[Proof of Theorem~\ref{teo.C1.intro}.]
Let us start showing that $u^*_\alpha$ is differentiable inside
$\Omega$. From our previous result, 
Lemma \ref{lema-supporting-unique}, we know that the $\alpha$-convex envelope $u^*_\alpha$ has a unique supporting $\alpha$-hyperplane at every point.
Then, the natural candidate to be the tangent plane for $u^*_\alpha$
at $z \in \Omega$ is $u^*_\alpha (z) + \langle z-z_0,  D  \pi_z (z) \rangle$
with $\pi_z$ is the supporting $\alpha$-hyperplane at $z$. 
Since $\pi_z$ is differentiable at $z$ we have
$$
\begin{array}{l}
\displaystyle 
u^*_\alpha (x) - u^*_\alpha (z) - \langle x-z,  D  \pi_z (z) \rangle
\\[7pt]
\displaystyle \qquad \geq \pi_z (x) - \pi_z (z) - \langle x-z,  D  \pi_z (z) \rangle
= o (|x-z|)
\end{array}
$$
as $x \to z$.

To show that 
$$
\begin{array}{l}
\displaystyle 
u^*_\alpha (x) - u^*_\alpha (z) - \langle x-z,  D  \pi_z (z) \rangle
\leq o (|x-z|)
\end{array}
$$
as $x \to z$, we argue by contradiction. 
Assume that $\{x_n\}$ is such that $x_n\to z$ and
\begin{equation} \label{kkk}
u^*_\alpha (x_n) - u^*_\alpha (z) - \langle x_n-z,  D  \pi_z (z) \rangle
\geq c |x_n-z|.
\end{equation}
Then, take $\{\pi_{x_n}\}$ the sequence of supporting $\alpha$-hyperplanes at $x_n$.
 We have that
 \begin{equation} \label{eq.jy}
 u^*_\alpha (y) \geq \pi_{x_n} (y)\qquad \mbox{ and } \qquad
 u^*_\alpha (x_n) = \pi_{x_n} (x_n).
 \end{equation}
 Since $u^*_\alpha$ is Lipschitz we can extract a subsequence 
that we still denote by $\{\pi_{x_n}\}$ that converges uniformly in a neigbourhood
of $z$ to an $\alpha$-hyperplane
$\pi_0$ as $n \to \infty$. Passing to the limit in \eqref{eq.jy} we get
$$
u^*_\alpha (y) \geq \pi_{0} (y)\qquad \mbox{ and } \qquad
 u^*_\alpha (z) = \pi_{0} (z).
$$
Hence, the limit $\pi_0$ is a supporting $\alpha$-hyperplane at $z$.
But now we observe that, using \eqref{kkk}, we have $\pi_0 \neq \pi_z$
a contradiction with the fact that the supporting $\alpha$-hyperplane at $z$
is unique. 
This proves that
$$
\begin{array}{l}
\displaystyle 
u^*_\alpha (x) - u^*_\alpha (z) - \langle x-z,  D  \pi_z (z) \rangle
= o (|x-z|)
\end{array}
$$
as $x \to z$ and then $u^*_\alpha $ is differentiable at $z$ (and its 
gradient is just the gradient of the supporting $\alpha$-hyperplane at $z$). 

To end the proof we observe that the supporting $\alpha$-hyperplane 
changes continuously. 
To see this fact we can argue as before, take a convergent sequence $x_n \to z$ and let
$\{\pi_{x_n}\}$ be the sequence of supporting $\alpha$-hyperplanes at $x_n$. As before, we have that $\pi_{x_n}\to \pi_0$ (first along a subsequence and since the limit is unique eventually for all) and we conclude that the limit
$\pi_0 $ must be the supporting $\alpha$-hyperplane at $z$, $\pi_z$ (here we use again
the uniqueness result, Lemma  \ref{lema-supporting-unique}). This shows that the
supporting $\alpha$-hyperplane varies continuously: in particular observe that by the explicit form of supporting hyperplanes, $\pi_{x_n}\to \pi_0$ implies that also the gradients converge.

Now, we observe that the gradient of the $\alpha$-convex envelope
$u^*_\alpha$ coincides with the gradient of the supporting $\alpha$-hyperplane at $z$
that varies continuously and we conclude that $u^*_\alpha\in C^1(\Omega)$.
\end{proof}

\end{document}